\documentclass{amsart}
\usepackage[left=2cm,right=2cm,top=2cm,bottom=1.5cm,bindingoffset=0cm]{geometry}

\usepackage{amsthm}
\usepackage{graphicx,textcomp}
\usepackage{amsmath,bm,amsfonts,mathrsfs,amssymb}
\usepackage{hyperref}
\newtheorem{theorem}{Theorem}[section]
\newtheorem{lemma}[theorem]{Lemma}

\newtheorem{rem}[theorem]{Remark}

\usepackage{appendix}
\usepackage[normalem]{ulem}
\usepackage{graphicx,color}
\usepackage{xcolor}

\begin{document}
\title{On a polygon version of Wiegmann-Zabrodin formula }
\author{Alexey Kokotov}
\address{Department of Mathematics \& Statistics, Concordia University, 1455 De Maisonneuve Blvd. W. Montreal, QC  H3G 1M8, Canada, \url{https://orcid.org/0000-0003-1940-0306}}
\email{alexey.kokotov@concordia.ca}
\thanks{The research of the first author was supported by Max Planck Institute for Mathematics in Bonn}

\author{Dmitrii Korikov}
\address{Department of Mathematics \& Statistics, Concordia University, 1455 De Maisonneuve Blvd. W. Montreal, QC  H3G 1M8, Canada, \url{https://orcid.org/0000-0002-3212-5874}}
\email{dmitrii.v.korikov@gmail.com}
\thanks{The research of the second author was supported by Fonds de recherche du Qu\'ebec.}


\begin{abstract}
Let $P$ be a convex polygon in ${\mathbb C}$ and let $\Delta_{D, P}$ be the operator of the Dirichlet boundary value problem for the Lapalcian $\Delta=-4\partial_z\partial_{\bar z}$ in $P$.
We derive a variational formula for the logarithm of the $\zeta$-regularized determinant of $\Delta_{D, P}$ for arbitrary infinitesimal deformations of the polygon $P$ in the class of polygons (with the same number of vertices). For a simply connected domain with smooth boundary such a formula was recently discovered by Wiegmann and Zabrodin as a non obvious corollary of the Alvarez variational formula, for domains with corners this approach is unavailable (at least for those deformations that do not preserve the corner angles) and we have to develop another one.             
\end{abstract}

\maketitle

\section{Introduction} The $\zeta$-regularized determinants of the operators of the Dirichlet boundary value problems for the Laplacian in domains with smooth boundaries in ${\mathbb C}$ or,
more generally, in a two-dimensional Riemannian manifold have been appearing in numerous works both as a subject and a technical tool (see, e. g., \cite{Alvarez}, \cite{OPS}, \cite{OPS1}, \cite{Kim}, \cite{BFK} to mention a few).  The basic fundamental result here is the Alvarez comparison formula for the determinants of the Dirichlet boundary problem for the Laplacians corresponding to two conformally equivalent metric (as well as its infinitesimal version describing the variation of the determinant under a {\it smooth}  variation of the conformal factor). 

 The non-smooth case also attracted some attention. First, one has to mention the explicit formula for the determinant of the Dirichlet boundary value problem in a polygon proposed in 1994 by Aurell and Salomonson (\cite{AS}). Their methods are partially heuristic. 

It can be shown that the rigorous variational formula for the determinant derived in the present work agrees with the Aurell–Salomonson result; moreover, it can be extended to provide a complete proof of that result. We omit the proofs of these latter two statements here, as they are straightforward and can be readily generated using automated AI assistance.

Recently there were established several versions of the Alvarez  formula for {\it smooth} conformal variations of conical metrics on closed manifolds and curvilinear polygonal domains (\cite{Aldana}, \cite{PeltolaWang}, \cite{Krusell}); let us also point to an interesting formula for the variation of the determinant of the Dirichlet boundary value problem in a round sector under variation of the corner angle (\cite{Aldana}, Theorem 1.8), the latter result (but not the methods of \cite{Aldana}) might have some relations with the present work of ours.         

The goal of the present work is to establish a variational formula for the determinant of the Dirichlet boundary value problem for the Laplacian in a polygon under general deformations of the latter (as a polygon: say, an infinitesimal shift of a vertex). 

Let us emphasize that a general deformation does not preserve the angles of the polygon and, therefore, the conformal factor relating the initial operator and the perturbed one is {\it not smooth}, thus, one cannot apply here the versions of the Alvarez formula from \cite{Aldana}, \cite{PeltolaWang}, \cite{Krusell}. 
Moreover, had a needed {\it non-smooth} version of the Alvarez formula existed it would be problematic to use it: even for the case of a simply connected domain with smooth boundary it is very nontrivial to derive from the classical Alvarez formula the formula of the variation of the determinant of the Dirichlet boundary problem under smooth variations of the boundary: this was done only recently by Wiegmann and Zabrodin (see \cite{WZ}, formula (4.29)), we are discussing their result in a small appendix to the present paper.                 
 
To establish our result we use another approach: in a sense, we return to the starting point of the derivation of the classical Alvarez formula, i. e. a variational formula for an individual eigenvalue. In our case it is not the standard Feynman formula from the perturbation theory but an Hadamard's type geometric one. Then we use a machinery heavily based on a very specific property of a domain with straight edges: its Schottky double (the Riemann sphere) possesses a natural flat conical metric symmetric under antiholomorphic involution, this gives us a chance to apply the methods we used in \cite{KokKor1}, \cite{KokKor2}. 

Notice that the assumption of the convexity of the polygon is a bit redundant (and is made only for brevity sake):  our variational formula holds  for the shifts of the vertices of the polygon with corresponding angles less than $\pi$, the fixed vertices may have arbitrary angles. At the moment we do not know how essential this requirement is.  

Let us emphasize that our variational formula looks as a very natural counterpart to Wiegmann-Zabrodin's one (hence, the title of the paper) although the derivations of these two formulas are based on different ideas.  

\section{Hadamard formula and formulation of the main result}
\subsection*{Hadamard formula for polygon.} 
The following version of the Hadamard formula for the Dirichlet problem is the starting point of the present work.

Let $P$ be a convex polygon  in $\mathbb{C}$ with vertices $x_1,\dots,x_n$ of angles $\alpha_1,\dots,\alpha_n$. Consider a family of smooth deformations $P_t:=\phi_t(P)$ of the polygon $P$, where $(-t_0,t_0)\ni t\mapsto\phi_t$ is a one-parametric smooth family of  diffeomorphisms $\phi_t: \mathbb{C}\mapsto\mathbb{C}$ such that $\phi_0$ is the identity. Notice that all the deformations preserving $P_t$ in the class of polygons (with the same number of vertices) belong to this family: such a deformation can be locally realized as a (real) linear map (shift+anisotropic scaling) in $t,\Re x,\Im x$ near vertex/sides; the latter can be extended to a smooth family $t\mapsto\phi_t$ to the whole $\mathbb{C}$. 

The above deformations can be exemplified as follows. Let $a=x_i$ and $b=x_{i+1}$. Assume that the origin coincides with the intersection point of the lines containing the sides of $P$ adjacent to $[a,b]$. Let $\chi$ be a cut-off function supported in a small neighborhood of $[a,b]$ and equal to one in a smaller neighborhood of $[a,b]$. Then the rule 
$$\phi_t(x):=x(1-t\chi(x))$$ 
defines the smooth family of diffeomorphisms for small $t$ which performs parallel shift of $[a,b]$. More generally, the family
\begin{equation}
\label{shif plus rot}
x\mapsto \phi_t(x):=x-\Big[x c_1+\frac{e^{i\alpha_i}\Re(e^{i\alpha_i}x)}{e^{-2i\alpha_{i+1}}-e^{2i\alpha_i}}c_2 \Big]\chi(x)t
\end{equation}
makes parallel shift and rotation of the side $[a,b]$.

Introduce the shift vector $A_t$ by $A_t(\phi_t(x))=\partial_t\phi_t(x)$. Let 
$$\Delta^{(t)}_{D}:\,u\mapsto -4\partial_x\partial_{\overline{x}}u \qquad (u\in C_c^{\infty}(P_t))$$ 
be the Friedrichs extension of the Dirichlet Laplacian, let \[\lambda_1(t)\le\lambda_2(t)\le\dots\] be the sequence of its eigenvalues (counted with their multiplicities) and let \[u_1^{(t)},u_2^{(t)}\dots\] be the corresponding orthonormal basis of eigenfunctions. Under the above assumptions, the following version of the classical Hadamard formula
\begin{equation}
\label{Hadamard}
\delta\lambda_j=-\int\limits_{\partial P}(\partial_\nu u_j)^2\,(A\cdot\nu)dl,
\end{equation}
holds true for the variation $\delta\lambda_j=(d\lambda_j/dt)|_{t=0}$ of a simple eigenvalue $\lambda_j$ (see Theorem 4.5, \cite{COX}). Here $u_k=u_k^{(0)}$, $A\equiv A_0$ and $\nu$ is the outward normal. (If  the eigenvalue $\lambda_j=\lambda_j(0)=\lambda_{j+1}(0)=\dots=\lambda_{l}(0)$ is not simple, then the summation over all coinciding eigenvalues should be applied to the both sides of (\ref{Hadamard}), i.e.
$$\delta\Big(\sum_{k=j}^{l}\lambda_{l}\Big)=-\int\limits_{\partial P}\sum_{k=j}^{l}(\partial_\nu u_k)^2\,(A\cdot\nu)dl=\int\limits_{\partial P}\Big(\underset{\lambda=\lambda_j}{\rm res}\partial_{\nu(x)}\partial_{\nu(y)}R^D_\lambda(x,y)\Big)\Big|_{y=x}(A(x)\cdot\nu(x))dl(x),$$
where $R^D_\lambda$ is the resolvent kernel of $\Delta_D$.
Note that the right-hand side does not depend on the choice of the orthonormal basis $u_{j}=u_j^{(0)},\dots,u_l=u_l^{(0)}$ in the eigenspace corresponding to $\lambda_j$. It should be noticed that, according to the result from \cite{Luc}, generically the spectrum of $\Delta_D$ is simple).

Without delving into the machinery of the theory of elliptic boundary value problems in non-smooth domains (see, e.g., \cite{NP}), we only note that the convexity condition for $P$ used in the justification of (\ref{Hadamard}) in \cite{COX} ensure that the eigenfunctions $u_j$ and their partial derivatives vanish at corners.

\subsection*{The main result: variational formula for the determinant of the Dirichlet boundary value problem in a polygon.}
By gluing two copies, $P\equiv P\times\{+\}$ and $P\times\{-\}$, of the polygon $P$ along the boundaries, one obtains a polyhedral surface $X$ endowed with  involution $x\mapsto x^\dag$ interchanging the points $(x,\pm)$. The Schwarz–Christoffel map $x: \mathbb{C}_+\to\Pi$, 
\begin{equation}
\label{Shwarz}
z\mapsto x(z):=C\int\limits_\cdot^z \prod_{k=1}^n(\zeta-z_k)^{\frac{\alpha_k}{\pi}-1}d\zeta.
\end{equation}
extends to the isometric biholomorphism from $\overline{\mathbb{C}}$ onto $X$ by symmetry $x(\overline{z})=x(z)^\dag$, where the (pullback) metric on $\overline{\mathbb{C}}$ is given by \[\mathfrak{T}(z)=\big|C\prod_{k=1}^n (z-z_k)^{\frac{\alpha_k}{\pi}-1}dz\big|^2\,.\]

Recall that the $\zeta$-function $s\mapsto\tilde{\zeta}_{\Delta^{(t)}_{D}-\mu}(s)$ of the operator $\Delta^{(t)}_{D}-\mu$ ($\mu\le 0$) as the meromorphic continuation of the sum $\sum_k(\lambda_k(t)-\mu)^{-s}$. It is well-known that $s\mapsto\tilde{\zeta}_{\Delta^{(t)}_{D}-\mu}(s)$ such continuation exists and is holomorphic outside the simple poles $s=1$, { $s=1/2-k$ ($k=0,1,\dots$)}. Introduce the spectral determinant of the operator $\Delta^{(t)}_{D}$ by
\[{\rm det}\Delta^{(t)}_{D}={\rm exp}\big(-\partial_s\zeta_{\Delta^{(t)}_{D}}(s)\big|_{s=0}\big).\]
The main result of the paper is the formula 
\begin{equation}
\label{Main formula}
\delta{\rm log}{\rm det}\Delta_{D}=\frac{1}{6\pi}\Im\mathscr{H}\int\limits_{\partial P}\{z,x\}(A\cdot\nu)\hat{\nu}(x)dx+\frac{{\gamma}}{12}\sum_{i=1}^n\Big(\frac{\pi}{\alpha_i}-\frac{\alpha_i}{\pi}\Big)\,\frac{\delta\alpha_i}{\alpha_i}{ -\sum_{i=1}^n\mathscr{K}(\alpha_i)\delta\alpha_i},
\end{equation}
for the variational derivative $\delta{\rm log}{\rm det}\Delta_{D}:=\frac{d}{dt}{\rm log}{\rm det}\Delta^{(t)}_{D}\big|_{t=0}$ under  deformations preserving the class of polygons. Here $\{\cdot,\cdot\}$ is the Schwarzian derivative, $\hat{\nu}(x)=-idx/dl(x)$ is the complexified outward normal, and $\mathscr{H}$ denotes the Hadamard regularization
\begin{equation}
\label{Hadamard reg def}
\mathscr{H}\int\limits_{\partial P}f(x)dx:=\lim_{\epsilon\to 0}\Bigg[\int\limits_{\partial P\backslash U(\epsilon)}f(x)dx-\sum_{k,j}c_{kj}\epsilon^{e_k}({\rm log}\epsilon)^{\tilde{e}_j}\Bigg],
\end{equation}
where $U(\epsilon)$ is the $\epsilon$-neighborhood of the set $\{x_1,\dots,x_{n}\}$ of vertices, the coefficients $c_{kj}\in\mathbb{C}$ are chosen in such a way that the limit in the right-hand side of (\ref{Hadamard reg def}) exists and is finite and the sum in the right-hand side of (\ref{Hadamard reg def}) contains only singular (growing as $\epsilon\to 0$) terms {(the same regularization applies to integrals over $(0,\infty)$, with a lower cutoff) this definition remains if the integration is over semi-axis $[0,+\infty)$ with cut-off near $0$). The function $\mathscr{K}$ is given by
\begin{align}
\label{mathscr K}
\mathscr{K}(\alpha)=\frac{\pi}{8\alpha^2}\mathscr{H}\int_{0}^{+\infty}\frac{dy}{{\rm tanh}(y/2){\rm sinh}^2(\pi y/2\alpha)}+\Big(\frac{\pi}{\alpha}-\frac{\alpha}{\pi}\Big)\frac{{\rm log}(2)}{12\alpha}+\frac{1}{24\pi}.
\end{align}}

The rest of the paper is organized as follows. In Section \ref{sec bumping}, we give the formal derivation of (\ref{Main formula}) for smooth deformations vanishing near vertices. In Sections \ref{sec shift} and \ref{sec rot}, we derive formally the same formula for deformations that coincide near one vertex with parallel shift or a rotation of one of the sides and are identities far from this vertex. Since any infinitesimal deformation of $P$ preserving the class of polygons can be represented as a sum of such deformations, the above calculations lead to (\ref{Main formula}). The only step skipped in the proof of (\ref{Main formula}) in the above sections is the justification of interchanging the analytic continuation and the differentiation in $t$ in the formula
$$(s-1)\dots(s-1-Q)\delta\zeta_{\Delta}(s)=\frac{(Q+1)!}{2\pi i}\int\limits_\Gamma\delta\zeta_{\Delta-\mu}(2+Q)\mu^{1+Q-s}d\mu$$
with large $Q$. It is worth noting that such a justification is traditionally absent in the works where the above formula is used in deriving the determinants (e.g., in \cite{KokKorot}). For simplicity, we made the justification of this formula in Section \ref{justy} for the class of polygon-preserving deformations only. Finally, in a short Appendix we discuss the smooth version of (\ref{Main formula}) belonging to Wiegmann and Zabrodin.

\section{Bending the side far from vertices}
\label{sec bumping}
\subsection*{Variation of $\zeta_{\Delta_D-\mu}(2+Q)$.} Suppose that $(A\cdot\nu)$ vanishes near vertices and introduce the domain $\Omega=\Omega(t)$ with smooth boundary obtained by smoothing corners of $P$ and obeying ${\rm supp}(A\cdot\nu)\subset\partial\Omega\cap\partial P$. In view of (\ref{Hadamard}), we have
\begin{align}
\label{variation of -2 term}
\delta(\lambda_j-\mu)^{-(2+Q)}=\frac{-\delta\lambda_j}{(1+Q)!}[\partial^{2+Q}_\mu(\lambda_j-\mu)^{-1}]=\frac{1}{(1+Q)!}\int\limits_{\partial P}\partial_\mu^{2+Q}\Big[\frac{\partial_{\nu(x)} u_j(x)\partial_{\nu(y)} u_j(y)}{\lambda_j-\mu}\Big]_{y=x}\,(A\cdot\nu)dl(x)
\end{align}
where $\lambda_j\equiv\lambda_j(0)$. Making the summation over $j$, one obtains 
\begin{align}
\label{variation of zeta 2 mu}
\delta\zeta_{\Delta_D-\mu}(2+Q)=\frac{1}{(1+Q)!}\int\limits_{\partial P}\partial_{\nu(x)}\partial_{\nu(y)}[\partial_\mu^{2+Q} R^D_\mu(x,y)]\big|_{y=x} \, (A\cdot\nu)dl(x),
\end{align}
where $R^D_\mu$ is the resolvent kernel of $\Delta_D$. 

To justify the interchanging of summation in $j$ and the differentiation in $t$ in (\ref{variation of zeta 2 mu}) for sufficiently large $Q$, one needs to prove that the series $\sum_j \delta(\lambda_j-\mu)^{-(2+Q)}$ converge uniformly in $t$ for such $Q$. To this end, one applies the Weyl's law ${\lambda_j=4\pi j/|P|}$ (uniformly in $t$) and the estimates
\begin{equation}
\label{estimates of eigenderivatives}
|\delta\lambda_j|\le C(A)\|\partial_\nu u_j\|^2_{H^{1/2}(\partial\Omega)}\le C(A,\Omega)\|u_j\|^2_{H^{2}(\Omega)}\le C(A,\Omega)[\|\Delta u_j\|^2_{L_2(\Omega)}+\|u_j\|^2_{L_2(\Omega)}]\le C(A,\Omega)\big(\lambda_j^2+1\big)
\end{equation}
following from (\ref{Hadamard}), the Sobolev trace theorem, and the local increasing smoothness theorems for solutions to elliptic boundary value problems. Here $C(A)$, $C(A,\Omega)$ are bounded unifromly in $t$. Thus, $$|\delta(\lambda_j-\mu)^{-(2+Q)}|=O(j^{-1-Q})$$ 
uniformly in $t$ and (\ref{variation of zeta 2 mu}) is valid at least for $Q=1,2,\dots$.

Let $\Delta$ be the Friedrichs extension of the Laplacian $-4\mathfrak{T}^{-1}\partial\overline{\partial}$ on $X$. In view of the symmetry of $X$, $R^D_\mu(x,y)$ is obtained by the anti-symmetrization of the resolvent kernel $R_\mu(x,y)=R_\mu(x^\dag,y^\dag)$ of $\Delta$,
\begin{equation}
\label{Dirichlet via polyhedra}
R^D_\mu(x,y)=R_\mu(x,y)-R_\mu(x,y^\dag), \qquad x,y\in X.
\end{equation}
Since both sides of (\ref{Main formula}) are additive in $A$ and invariant under translations and rotations of the complex plane, one can assume, without loss of generality, that $(A,\nu)$ is supported on the one side of $P$ which coincides with the segment $[x_1,x_2]$ of the real axis. Let $x$ be extended to the coordinates on $X\backslash P$ by the Schwarz reflection principle $x^\dag=\overline{x}$. Then the near-diagonal asymptotics
\begin{equation}
\label{resolvent near d}
R_\mu(x,y)=-\frac{{\rm log}r}{2\pi}\Big(1-\frac{\mu r^2}{4}\Big)+\tilde{R}_\mu(x,y), \qquad r:={\rm dist}(x,y)
\end{equation}
is valid and admits differentiation, where $\tilde{R}_\mu$ is twice continuously differentiable in each variable outside the vertices. Since the first term in the right-hand side of (\ref{resolvent near d}) is linear in $\mu$, we have \[\partial_{\nu(x)}\partial_{\nu(y)}[\partial_\mu^2 R^D_\mu(x,y)]\big|_{y=x}=\partial_\mu^2\psi_\mu(x),\] 
where
\begin{equation}
\label{Psi}
\psi_\mu(x):=\partial_{\nu(x)}\partial_{\nu(y)}[\tilde{R}_\mu(x,y)-\tilde{R}_\mu(x,\overline{y})]\big|_{y=x}.
\end{equation}
Note that $\psi_\mu$ is holomorphic in $\mu$ outside the spectrum of $\Delta_D$. Now (\ref{variation of zeta 2 mu}) can be rewritten as 
\begin{align}
\label{variation of zeta 2 mu short}
(1+Q)!\,\delta\zeta_{\Delta_D-\mu}(2+Q)=\int\limits_{\partial P}\partial_\mu^{2+Q}\psi_\mu(x)\,(A\cdot\nu)dl(x).
\end{align}
The function
\begin{equation}
\label{Psi big 1}
\Psi_\mu:=\int\limits_{\partial P}(A\cdot\nu)\,\psi_\mu(x)dl(x);
\end{equation}
is well-defined if the support of $(A\cdot\nu)$ is separated from vertices, and (\ref{variation of zeta 2 mu short}) takes the form 
$$(1+Q)!\,\delta\zeta_{\Delta_D-\mu}(2+Q)=\partial_\mu^{2+Q}\Psi_\mu.$$

\subsection*{Expression for $\psi_0$.} Recall the Verlinde-Verlinde formula (see \cite{VerVer,KokKor1})
\begin{equation}
\label{Verlinde}
G(x,y)=\frac{1}{2}(m(x)+m(y))-\frac{1}{4\pi}{\rm log}\big(|E(x,y)|^2\sqrt{\mathfrak{T}(x)\mathfrak{T}(x)}\big).
\end{equation}
for the Green function for the Friedrichs Laplacian $\Delta$ on $X$; here $E$ is the prime form on $X$, the Robin's mass $m$ is a smooth function outside vertices (see, e.g., \cite{KokKor1}), and the metric is given by $\mathfrak{T}(x)=\mathfrak{T}(x)=1$ in the coordinates $x$. In view of (\ref{Verlinde}) and the Laurent expansion 
\[R_\mu(x,y)={-{\rm Area}(X)^{-1}\mu^{-1}}+G(x,y)+O(\mu)\]
for the resolvent, formulas (\ref{Psi}), (\ref{resolvent near d}) imply
\begin{align*}
\psi_0(x)=\partial_{\nu(x)}\partial_{\nu(y)}\Big[G(x,y)&-G(x,\overline{y})+\frac{1}{2\pi}{\rm log}\Big|\frac{x-y}{x-\overline{y}}\Big|\Big]_{y=x}=\\
=&\frac{1}{2\pi}\partial_{\nu(x)}\partial_{\nu(y)}{\rm log}\Big|\frac{x-y}{E(x,y)}\frac{E(x,\overline{y})}{x-\overline{y}}\Big|\Big|_{y=x}.
\end{align*}
The near-diagonal asymptotics of the prime form is given (in any holomorphic coordinates on $X$) by
\begin{equation}
\label{prime form asymp}
\frac{y-x}{E(x,y)}=1+\frac{1}{12}S_0(x)(y-x)^2+O((y-x)^3)
\end{equation}
(see formula (1.3), \cite{Fay}), where the holomorphic projective connection $S_0$ transforms as 
\begin{equation}
\label{proj connection}
S_0(x)=S_0(z)(\partial z/\partial x)^2+\{z,x\}.
\end{equation}
Note that the prime form is given by 
\[E(z,z')=\frac{z'-z}{\sqrt{dzdz'}}\]
in the (global) holomorphic coordinates $z\in\overline{\mathbb{C}}\equiv X$. Thus, (\ref{prime form asymp}) and (\ref{proj connection}) yield $S_0(z)=0$ and
\begin{equation}
\label{connection via shwarz}
S_0(x)=\{z,x\}.
\end{equation}
In view of (\ref{prime form asymp}) and (\ref{connection via shwarz}), we have 
\begin{align*}
{\rm log}\Big|\frac{x-y}{E(x,y)}\frac{E(x,\overline{y})}{x-\overline{y}}\Big|=\Re\Big[\frac{S_0(x)}{12}[(y-x)^2-(\overline{y}-x)^2]\Big]+O(|x-y|^3)
\end{align*}
and
\begin{equation}
\label{Psi 0}
\psi_0(x)=-\frac{\Re S_0(x)}{6\pi}=-\frac{\Re\{z,x\}}{6\pi}.
\end{equation}

\subsection*{Asymptotics of $\psi_\mu$ at infinity.} The resolvent kernel $R_\mu$ admits the asymptotics
\begin{equation}
\label{resolvent asymptotics}
R_\mu(x,y)=\frac{\chi(r)}{2\pi}K_0(\sqrt{-\mu}r)+\breve{R}_\mu(x,y),
\end{equation}
where $\chi$ is a smooth cut-off function equal to $1$ near the origin, the support of $\chi$ is sufficiently small, $K_0$ is the Macdonald function, and the remainder $\breve{R}_\mu$ and all its derivatives decay exponentially as $\Re\mu\to -\infty$ and uniformly in $x,y\in X$ separated from vertices. The latter follows from the fact that
$$(\Delta-\mu)\breve{R}_\mu(x,y)=\frac{1}{2\pi}[\Delta,\chi(r)]K_0(\sqrt{-\mu}r)$$
and all its derivatives decays exponentially as $\Re\mu\to -\infty$, the operator estimate $(\Delta-\mu)^{-1}=O(|\mu|^{-1})$ as $\Re\mu\to-\infty$, and the Sobolev embedding theorems.

Comparison of formulas (\ref{resolvent asymptotics}) and (\ref{resolvent near d}) yields
\begin{equation}
\label{breve tilde}
\tilde{R}_\mu(x,y)=\breve{R}_\mu(x,y)-(1/2\pi)+({\rm log}(\sqrt{-\mu}/2)+\gamma-1)(\mu r^2/4-1)(1/2\pi){ +O(r^4{\rm log}r)}
\end{equation}
for $y$ sufficiently close to $x$. In particular, formula (\ref{Psi}) can be rewritten as
\begin{equation}
\label{Psy asymp infinity}
\psi_\mu(x)=[1+{\rm log}2-(1/2){\rm log}(-\mu)-\gamma]\mu/2\pi+\breve{\psi}_\mu(x), \qquad \partial_\mu^2\psi_\mu(x)=-\frac{1}{4\pi\mu}+\partial_\mu^2\breve{\psi}_\mu(x)
\end{equation}
where $\breve{\psi}_\mu(x){ :=\partial_{\nu(x)}\partial_{\nu(y)}[\breve{R}_\mu(x,y)-\breve{R}_\mu(x,\overline{y})]\big|_{y=x}}$ and all its derivatives decay exponentially and uniformly in $x\in\partial\Omega\cap\partial P$ as $\Re\mu\to -\infty$. Formulas (\ref{Psi big 1}), (\ref{Psy asymp infinity}) and the expression 
\begin{equation}
\label{area var}
\delta{\bm|}P{\bm|}=\int\limits_{\partial P}(A\cdot\nu)dl(x)
\end{equation}
for the variation of the polygon area ${\bm|}P{\bm|}$ imply
\begin{equation}
\label{Psi big 2 as}
\Psi_\mu:=[1+{\rm log}2-(1/2){\rm log}(-\mu)-\gamma]\mu\frac{\delta{\bm|}P{\bm|}}{2\pi}+\breve{\Psi}_\mu,
\end{equation}
where the remainder $\breve{\Psi}_\mu$ and all its derivatives decay exponentially as $\Re\mu\to -\infty$.

\subsection*{Variation of $\zeta_{\Delta_D}(s)$: formal calculations.} Recall that $\zeta_{\Delta_D}(s)$ and $\zeta_{\Delta_D-\mu}(2+Q)$ are related via
\begin{equation}
\label{zeta s via zeta 2 mu}
(s-1)\dots(s-1-Q)\zeta_{\Delta}(s)=\frac{(Q+1)!}{2\pi i}\int\limits_\Gamma\zeta_{\Delta-\mu}(2+Q)\mu^{1+Q-s}d\mu;
\end{equation} 
here $\Gamma$ is the (counterclockwise) path enclosing the ray $(-\infty,0]$. Formally differentiating both sides of (\ref{zeta s via zeta 2 mu}) in $t$ and taking into account (\ref{variation of zeta 2 mu short}), one arrives at
\begin{align}
\label{variation of zeta s q}
(s-1)\dots(s-1-Q)\delta\zeta_{\Delta_D}(s)=\int\limits_{\Gamma}\frac{\mu^{1+Q-s}d\mu}{2\pi i}\partial_\mu^{2+Q}\Psi_\mu.
\end{align}
Due to (\ref{Psy asymp infinity}) one can integrate by parts in the right-hand side of (\ref{variation of zeta s q}) obtaining
\begin{align}
\label{variation of zeta s}
(s-1)\delta\zeta_{\Delta_D}(s)=\int\limits_{\Gamma}\frac{\mu^{1-s}d\mu}{2\pi i}\partial_\mu^{2}\Psi_\mu.
\end{align}

Now we make use of the following simple property of the inverse Mellin transform.
\begin{lemma}
\label{general lemma}
Let $F$ be a function holomorphic in some neighborhood of the negative semi-axis containing the curve $\Gamma$ admitting the asymptotics
\begin{equation}
\label{F asymp}
F(\mu)=\sum_{k=1}^K (F_k+\tilde{F_k}\mu{\rm log}(-\mu))\mu^{r_k}+\tilde{F}(\mu),
\end{equation}
as $\Re\mu\to -\infty$, where $r_k\in\mathbb{R}$, $F_k,\tilde{F_k}\in\mathbb{C}$, and the remainder obeys
\begin{equation}
\label{F tilde est}
|{\tilde{F}}(\mu)|=O(|\mu|^{\kappa}), \qquad |\mu\partial^k_\mu {\tilde{F}}(\mu)|=O(\mu^{\kappa})
\end{equation}
for some $\kappa<0$. Denote by $F(\infty)$ and $\tilde{F}(\infty)$ the constant term and the coefficient at ${\rm log}(-\mu)$ in {\rm(\ref{F asymp})}. Let $\widehat{F}$ be the analytic continuation of the integral
\begin{equation}
\label{mellin transfrom of F}
\widehat{F}(s):=\int\limits_{\Gamma}\frac{\mu^{1-s}d\mu}{2\pi i}\partial_\mu^2F(\mu)
\end{equation}
initially defined for sufficiently large $\Re s$. Then $\widehat{F}$ is holomorphic at $s=0$ and
\[\widehat{F}(0)=\tilde{F}(\infty), \qquad \partial_s\widehat{F}(0)=F(\infty)-\tilde{F}(\infty)-F(0).\]
In particular, the function $s\mapsto\eta(s):=\widehat{F}(s)/(s-1)$ obeys
\[\eta(0)=-\tilde{F}(\infty), \qquad -\partial_s\eta(0)=F(\infty)-F(0).\]
\end{lemma}
\begin{proof}
Suppose, for a while, that $\Re s$ is sufficiently large. We have
\begin{equation}
\label{general regularization 1}
\widehat{F}(s)=\pi^{-1}e^{-i\pi s}{\rm sin}(\pi s)J_\infty(s)+J_0(s),
\end{equation}
where
\[J_0(s):=\int\limits_{|\mu|=\epsilon}\mu^{1-s}\partial_\mu^2 F(\mu)\frac{d\mu}{2\pi i}, \qquad J_\infty(s):=\int\limits_{-\infty-i0}^{-\epsilon-i0}\mu^{1-s}\partial_\mu^2F(\mu) d\mu\]
and $\epsilon>0$ is sufficiently small. Note that $J_0$ is an entire function obeying $J_0(0)=0$ since the integration contour $|\mu|=\epsilon$ is compact. 

Integrating by parts, we obtain
\begin{align*}
J_\infty(s)=\mu^{-s}[\mu\partial_\mu+s-1]F(\mu)\Big|^{-\epsilon}_{-\infty}+s(s-1)\int\limits_{-\infty-i0}^{-\epsilon-i0}\frac{F(\mu)}{\mu^{s+1}}d\mu
\end{align*}
where the substitution at minus infinity vanishes for sufficiently large $\Re s$. Integrating both sides of (\ref{F asymp}) yields
\begin{align*}
\int\limits_{-\infty-i0}^{-\epsilon-i0}\frac{F(\mu)}{\mu^{s+1}}d\mu-\int\limits_{-\infty-i0}^{-\epsilon-i0}\frac{\tilde{F}(\mu)}{\mu^{s+1}}d\mu&=\\
=\sum_{k=1}^K \Big[\frac{F_k}{r_k-s}+&\frac{\epsilon[1+(s-r_k-1){\rm log}\epsilon]\tilde{F}_k}{(s-r_k-1)^2}\Big](-\epsilon-i0)^{r_k-s}.
\end{align*}
Here the second integral is well-defined and holomorphic in $s$ for $\Re s>\kappa$ due to (\ref{F tilde est}). Therefore the first integral admits the meromorphic continuation to the semi-plane $\Re s>\kappa$. The right-hand side admits the expansion
\[-\frac{\tilde{F}(\infty)}{s^2}-\frac{F(\infty)+\pi i\tilde{F}(\infty)}{s}+O(1)\]
near $s=0$. Therefore, $J_\infty$ admits meromorphic continuation to the neighborhood of the origin and 
\[J_\infty(s)=\mu^{-s}[\mu\partial_\mu+s-1]F(\mu)\Big|^{-\epsilon}+F(\infty)+\Big(\frac{1}{s}+\pi i-1\Big)\tilde{F}(\infty)+O(s)\]
near $s=0$. Now, (\ref{general regularization 1}) yields $\widehat{F}(0)=\tilde{F}(\infty)$.

Similarly, the integration by parts yields
\begin{align*}
-\partial_s J_0(s)=&\int\limits_{|\mu|=\epsilon}\frac{\mu^{1-s}\,{\rm log}\mu\,d\mu}{2\pi i}\partial_\mu^2 F(\mu)=\mu^{-s}\big[\mu\partial_\mu+s-1\big]F(\mu)\frac{{\rm log}\mu}{2\pi i}\Big|_{-\epsilon-i0}^{-\epsilon+i0}-\\
-&s\int\limits_{|\mu|=\epsilon}\mu^{-1-s}[(1-s){\rm log}\mu+1]\,F(\mu)\frac{d\mu}{2\pi i}+(1-s)\int\limits_{|\mu|=\epsilon}\mu^{-s-1}F(\mu)\frac{d\mu}{2\pi i}=\\
&=\mu^{-s}\big[\mu\partial_\mu+s-1\big]F(\mu)\Big|^{-\epsilon}+O(s)+F(0).
\end{align*}
Now, the differentiation of (\ref{general regularization 1}) and taking into account that ${\rm sin}(x)/x=1+O(x^2)$ leads to 
\begin{align*}
\partial_s\widehat{F}(s)=\Big[J_\infty(s)-\frac{\tilde{F}(\infty)}{s}\Big]_{s=0}+\partial_s J_0(0)-&i\pi\tilde{F}(\infty)+O(s)=\\
=&F(\infty)-\tilde{F}(\infty)-F(0)+O(s).
\end{align*}
\end{proof}
In view of (\ref{variation of zeta s}), the substitution of $F(\mu)=\Psi_\mu$, $\eta(s)=\delta\zeta_{\Delta_D}(s)$ into Lemma \ref{general lemma} leads to
\begin{equation}
\label{variation of logdeter}
\delta{\rm log}{\rm det}\Delta_{D}=0-\Psi_0.
\end{equation}
Substitution of (\ref{Psi big 1}) and (\ref{Psi 0}) into (\ref{variation of logdeter}) leads to (\ref{Main formula}).

\section{Parallel infinitesimal shifts of sides near vertices}
\label{sec shift}
\subsection*{Variation of $\zeta_{\Delta_D-\mu}(2+Q)$.} Now we consider the perturbation of $\partial P$ which coincides with a parallel shift of a side $[a:=x_i,x_{i+1}=:b]$ of the polygon near one of its ends (say, $a$) and vanishes near the other end. In other words, we assume that $(A\cdot\nu)$ is supported and is smooth on $[a,b]$, $(A\cdot\nu)=1$ in the segment $[a,q]$, and $(A\cdot\nu)=0$ near $b$. As in the previous paragraph, we can assume that $[a,b]\subset\mathbb{R}$. 

An example of the above variations is the family of diffeomorphisms
$$\phi_t:\,x\mapsto x+\upsilon t\chi(x-a) \qquad \Big(t\in(-t_0,t_0), \quad \upsilon:=\frac{a-x_{i-1}}{|a-x_{i-1}|}\Big),$$
where $\chi$ is a smooth positive cut-off function equal to 1 near the origin and to 0 outside a small neighborhood of the origin.

Since $(A\cdot\nu)$ does not vanish near one of the ends, one cannot apply here directly the calculations of the previous section to find the variation $\delta{\rm logdet}\Delta_D$: integral (\ref{Psi big 1}) does not converge (cf. (\ref{Psi 0}) for the case $\mu=0$). A more important problem in generalizing the calculations of Section \ref{sec bumping} is that now they require the near-diagonal asymptotics of the resolvent kernel to be uniform both in high negative $\mu$ and in small distance $r_0$ to the vertex. Although it is possible to obtain such an asymptotics from the exact Sommerfeld-type formulas for the heat kernel in a cone \cite{Carslaw,Dowker}, we present a simpler derivation based on the the deformation of the integration contour $[a,b]$ in (\ref{variation of zeta 2 mu}), (\ref{Psi big 1}) away from the vertex $a$.

Let us rewrite (\ref{Hadamard}) as
\begin{equation}
\label{to be shifted}
\delta\lambda_j=-\int\limits_{[q,b]}(\partial_\nu u_j)^2\,(A\cdot\nu)dl-\int\limits_{[a,q]}(\partial_\nu u_j)^2dl.
\end{equation}
Introduce the form
\begin{equation}
\label{closed shifting}
\omega_j=4(\partial_x u_j)^2dx-\lambda_j u_j^2d\overline{x}.
\end{equation} 
Note that $\omega_j$ is closed since 
\begin{align*}
d\omega_j=4\partial_{\overline{x}}(\partial_x u_j)^2\,d\overline{x}\wedge dx-\partial_{x}(\lambda_j u_j^2)\,dx\wedge d\overline{x}&=\\
=2\,[4\partial_{\overline{x}}\partial_x u_j+\lambda_j u_j]\partial_x u_j \,d\overline{x}\wedge dx&=-2(\Delta-\lambda_j)u_j\partial_{x}u_j\,d\overline{x}\wedge dx=0.
\end{align*}
At the same time, we have
\[\omega_j=-(\partial_{\nu(x)}u_j)^2 dl(x) \text{ on } (a,b)\]
due to the Dirichlet boundary conditions $u_j=\partial_l u_j=0$ and the equality $dx=dl(x)$ on $(a,b)$. Thus, the second integral in the right-hand side of (\ref{to be shifted}) can be rewritten as
\begin{equation}
\label{shift}
\int\limits_{[a,q]}(\partial_\nu u_j)^2dl=\frac{-1}{{\rm sin}\alpha}\Im\Big({ e^{i\alpha}}\int\limits_{[a,q]}\omega_j\Big)=\frac{-1}{{\rm sin}\alpha}\Im\Big(e^{i\alpha}\Big[\int\limits_{[a,q']}\omega_j+\int\limits_{\mathcal{C}}\omega_j\Big]\Big)
\end{equation}
with any $\alpha\in(0,\pi)$; here $[a,q']$ is a segment of the side $[x_{i-1},x_i=a]$ coming out from the same vertex $a$ as $[a,b]$ and $\mathcal{C}$ is an arbitrary curve in $\overline{P}\backslash\{a\}$ that starts at $q'$ and ends at $q$. 

Due to Dirichlet conditions $u_j=0$ on $[x_{i-1},a]\supset[q',a]$, formula (\ref{closed shifting}) implies
\begin{equation}
\label{choice of alpha}
\begin{split}
e^{i\alpha}\omega_j=4e^{i\alpha}(\partial_x u_j)^2dx=[y=e^{-i\alpha_i}x]=4e^{i\alpha}(e^{-i\alpha_i}\partial_y u_j)^2 e^{i\alpha_i}dl(x)&=\\
=-e^{i(\alpha-\alpha_i)}(\partial_{\nu(x)} u_j)^2dl(x) &\text{ on } [q',a],
\end{split}
\end{equation}
where $\alpha_i$ is the angle between $[a=x_i,x_{i+1}=b]$ and $[x_{i-1},x_i=a]$. In view of (\ref{choice of alpha}), the integral over $[a,q']$ makes no contribution into the right-hand side of (\ref{shift}) if $\alpha=\alpha_i$. Thus, we arrive at
\begin{equation}
\label{shift 1}
\delta\lambda_j=-\int\limits_{[q,b]}(\partial_\nu u_j)^2\,(A\cdot\nu)dl+\frac{1}{{\rm sin}\alpha_i}\Im\Big(e^{i\alpha_i}\int\limits_{\mathcal{C}}\omega_j\Big).
\end{equation}
Note that the integration in (\ref{shift 1}) is actually performed over the contour separated from vertices.

The substitution of (\ref{shift 1}) into the equation 
$$(1+Q)!\delta(\lambda_j-\mu)^{-(2+Q)}=-\delta\lambda_j\,\partial^{2+Q}_\mu(\lambda_j-\mu)^{-1}$$
yields
\begin{align*}
(1+Q)!\delta(\lambda_j-&\mu)^{-(2+Q)}-\int\limits_{[q,b]}\partial_\mu^{2+Q}\Big[\frac{\partial_{\nu(x)} u_j(x)\partial_{\nu(y)} u_j(y)}{\lambda_j-\mu}\Big]_{y=x}\,(A\cdot\nu)dl(x)=\\
=&\frac{-1}{{\rm sin}\alpha_i}\Im\Big(e^{i\alpha_i}\int\limits_{\mathcal{C}}\partial_{\mu}^{2+Q}\Big(\frac{\omega_j}{\lambda_j-\mu}\Big)\Big)=\frac{-1}{{\rm sin}\alpha_i}\Im\Big(e^{i\alpha_i}\int\limits_{\mathcal{C}}\partial_{\mu}^{2+Q}\Big[\frac{4(\partial_x u_j)^2\,dx-\mu u_j^2\,d\overline{x}}{\lambda_j-\mu}\Big]\Big)
\end{align*}
(cf. (\ref{variation of -2 term})). Making summation over $j$, one arrivers at 
\begin{align}
\label{variation of zeta 2 mu shift}
(1+Q)!\delta\zeta_{\Delta_D-\mu}(2+Q)=\partial^{2+Q}_\mu\Psi_\mu-\frac{1}{{\rm sin}\alpha_i}\Im\Big[e^{i\alpha_i}\int\limits_{\mathcal{C}}\partial^{2+Q}_\mu(4\partial_x\partial_y R^D(x,y)\,dx-\mu R^D(x,y)\,d\overline{x})\Big|_{y=x}\Big],
\end{align}
where $\Psi_\mu=\Psi_\mu([q,b])$ is obtained by replacing the integration contour in (\ref{Psi big 1}) by $[q,b]$. The interchanging the differentiation in the parameter and the summation over $j$ for $Q=1,2,\dots$ is justified in the same way as in (\ref{variation of zeta 2 mu}); the only difference is that $\partial\Omega$ in estimate (\ref{estimates of eigenderivatives}) is replaced by $[q,b]\cap\partial\Omega$ or by $\mathcal{C}\subset\Omega$.

Introduce the functions
\begin{align}
\label{Phi 0}
\begin{split}
\phi^{+}_\mu(x):=4\Big[\partial_{x}\partial_y\Big(R^D_\mu(x,y)&+\frac{{\rm log}r}{2\pi}\Big(1-\frac{\mu r^2}{4}\Big)\Big)\Big]_{y=x}=\\
&=4\partial_{x}\partial_y\big[\tilde{R}_\mu(x,y)-R_{\mu}(x,y^\dag)\big]_{y=x},
\end{split}
\end{align}
\begin{align}
\label{Phi 1}
\phi^{-}_\mu(x):=\mu\Big[R^D(x,y)+\frac{{\rm log}r}{2\pi}\Big]_{y=x}=\mu[\tilde{R}_\mu(x,x)-R_{\mu}(x,x^\dag)]
\end{align}
(here $r:={\rm dist}(x,y)$ and formula (\ref{Dirichlet via polyhedra}) is used). Since $\partial_{x}\partial_y{\rm log}|x-\overline{y}|=0$, near-diagonal asymptotics (\ref{resolvent near d}) for $R_\mu(x,y^\dag)$ implies that $\phi^{+}_\mu$ and $\phi^{-}_\mu$ grow logarithmically near the sides of the polygon. Thus, the following integrals are well defined
\begin{align}
\label{Phi 01}
\Phi_\mu(a,q):=\frac{-1}{{\rm sin}\alpha_i}\Im\Big(e^{i\alpha_i}\int\limits_{\mathcal{C}}\big[\phi^{+}_\mu(x)dx-\phi^{-}_\mu(x)d\overline{x}\big]\Big).
\end{align}
Now, (\ref{variation of zeta 2 mu shift}) can be rewritten as
\begin{equation}
\label{variation of zeta 2 mu shift 1}
(1+Q)!\delta\zeta_{\Delta_D-\mu}(2+Q)=\partial^{2+Q}_\mu[\Psi_\mu([q,b])+\Phi_\mu(a,q)].
\end{equation}

\subsection*{Expression for $\Phi_0(a,q)$.} From (\ref{Phi 1}), it is clear that $\phi^{-}_0\equiv 0$. Since the prime-form $E(x,y)$ is holomorphic in both variables and the involution $y\mapsto y^\dag$ is anti-holomorphic, we have 
\[2\partial_{x}\partial_{y}{\rm log}|E(x,y^\dag)|=\partial_y[\partial_x E(x,y^\dag)/E(x,y^\dag)]=0.\]
Therefore, formulas (\ref{Phi 0}), (\ref{Dirichlet via polyhedra}), (\ref{Verlinde}), (\ref{prime form asymp}), (\ref{connection via shwarz}) and the obvious relations $\partial_{x}\partial_{y}m(x)=\partial_{x}\partial_{y}m(y)=0$ imply
\begin{align*}
\begin{split}
\phi^{+}_0(x)&=4\partial_{x}\partial_{y}\Big(G(x,y)+\frac{{\rm log}r}{2\pi}-G(x,y^\dag)\Big)\Big|_{x=y}=\\=\frac{1}{\pi}&\partial_{x}\partial_{y}\,{\rm log}\Big|\frac{(x-y)E(x,y^\dag)}{E(x,y)}\Big|^2\Big|_{x=y}=\frac{\partial_{x}\partial_{y}[S_0(x)(y-x)^2]}{12\pi}\Big|_{x=y}=-\frac{\{z,x\}}{6\pi}.
\end{split}
\end{align*}
Thus, we have
\begin{equation}
\label{Phi at 0}
\Phi_0(a,q):=\frac{1}{6\pi\,{\rm sin}\alpha_i}\Im\Big(e^{i\alpha_i}\int\limits_{\mathcal{C}}\{z,x\}dx\Big).
\end{equation}

\subsection*{Asymptotics of $\Phi_\mu(a,q)$ at infinity.} Denote 
\[r_1:={\rm dist}(x,[a,b]), \qquad r_2:={\rm dist}(x,[x_{i-1},a]).\] In view of formulas (\ref{resolvent asymptotics}), (\ref{Dirichlet via polyhedra}), (\ref{breve tilde}) and the equalities
\[\partial_x\partial_y r^2=0, \qquad \big(4\partial_x\partial_y-\mu\big) K_0(\sqrt{-\mu}|x-\overline{y}|)=0\]
the functions
\begin{align*}
\phi^{+}_\mu(x)&+\frac{\mu}{2\pi}[K_0(2\sqrt{-\mu}r_1)+e^{-2i\alpha_i}K_0(2\sqrt{-\mu}r_2)], \\
\phi^{-}_\mu(x)&+\frac{\mu}{2\pi}\Big({\rm log}(\sqrt{-\mu}/2)+\gamma+K_0(2\sqrt{-\mu}r_1)+K_0(2\sqrt{-\mu}r_2)\Big)
\end{align*}
decay exponentially (and uniformly in $x$ separated from vertices) as $\Re\mu\to -\infty$. 

Suppose that the part $L_1$ (resp., $L_2$) of the contour $\mathcal{C}$ close to $[a,b]$ (resp., to $[x_{i-1},a]$) is a straight segment orthogonal to $\partial P$. Then  
\begin{align*}
\begin{split}
\Phi_\mu(a,q)\simeq&-\Im\Big(-\frac{\mu({\rm log}(\sqrt{-\mu}/2)+\gamma)}{2\pi\,{\rm sin}\alpha_i}e^{i\alpha_i}\int_{\mathcal{C}}d\overline{x}+\\
+&\frac{i\mu }{\pi\,{\rm sin}\alpha_i}\Big[e^{i\alpha_i}\int_{L_1}K_0(2\sqrt{-\mu}r_1)d\Im x+\int_{L_2}K_0(2\sqrt{-\mu}r_2)d\Im(e^{-i\alpha_i}dx)\Big]\Big)\simeq\\
&\simeq-\frac{\mu({\rm log}(\sqrt{-\mu}/2)+\gamma)}{2\pi\,{\rm sin}\alpha_i}\Im\big[(\overline{q'-q})e^{i\alpha_i}\big]-\frac{(1+{\rm cos}\alpha_i)}{4\,{\rm sin}\alpha_i}\sqrt{-\mu},
\end{split}
\end{align*}
where $\simeq$ denotes the equality up to the terms exponentially decaying (together with all derivatives) as $\Re\mu\to -\infty$. Here we used the equality $\int_{0}^{+\infty}K_0(t)dt=\pi/2$. Taking into account the expression
\begin{equation}
\label{perimeter variation}
\frac{1+{\rm cos}\alpha_i}{{\rm sin}\alpha_i}=\delta{\bm|}\partial P{\bm|}
\end{equation}
for the variational derivative of the boundary length ${\bm|}\partial P{\bm|}$ of the polygon, one finally arrives at
\begin{equation}
\label{asymptotics of Phi}
\begin{split}
\Phi_\mu(a,q)&=-\frac{1}{2}\big[\delta{\bm|}\partial P{\bm|}\sqrt{-\mu}+\pi^{-1}\mu({\rm log}(\sqrt{-\mu}/2)+\gamma)(q-a)\big]+\breve{\Phi}_\mu(a,q),\\
\partial_\mu^2\Phi_\mu(a,q)&=\frac{q-a}{4\pi(-\mu)}+{\frac{\delta{\bm|}\partial P{\bm|}}{8}}(-\mu)^{-3/2}+\partial_\mu^2\breve{\Phi}_\mu(a,q)
\end{split}
\end{equation}
where $\breve{\Phi}_\mu(a,q)$ and all its derivatives in $\mu$ decays exponentially as $\Re\mu\to -\infty$. Formulas (\ref{asymptotics of Phi}) and (\ref{Psy asymp infinity}) yield
\begin{equation}
\label{asymptotics of Phi plus Psi}
\begin{split}
\Psi_\mu([q,b])+\Phi_\mu(a,q)=\Big[1+{\rm log}2-(1/2){\rm log}(-\mu)-\gamma\Big]\frac{\mu}{2\pi}\delta{\bm|}P{\bm|}-\frac{\sqrt{-\mu}}{2}\delta{\bm|}\partial P{\bm|}+\frac{q-a}{2\pi}\mu+\breve{\Psi}_\mu([q,b])+\breve{\Phi}_\mu(a,q),\\
\partial_\mu^2[\Psi_\mu([q,b])+\Phi_\mu(a,q)]=\frac{\delta{\bm|}P{\bm|}}{4\pi(-\mu)}+{\frac{\delta{\bm|}\partial P{\bm|}}{8}}(-\mu)^{-3/2}+\partial_\mu^2[\breve{\Psi}_\mu([q,b])+\breve{\Phi}_\mu(a,q)]
\end{split}
\end{equation}
(Note that, since $\delta{\bm|}\partial P{\bm|}=0$ for smooth perturbations of $P$ vanishing near vertices, no terms proportional to $\delta{\bm|}\partial P{\bm|}$ arouse in the previous section.)

\subsection*{Variation of $\zeta_{\Delta_D}(s)$: formal calculations.} Put $F(\mu)=\Psi_\mu([q,b])+\Phi_\mu(a,q)$. Formally differentiating (\ref{zeta s via zeta 2 mu}) in $t$, applying (\ref{variation of zeta 2 mu shift 1}), and then integrating by parts in $\mu$, one arrives at $(s-1)\delta\zeta_{\Delta_D}(s)=\widehat{F}(s)$, where $\widehat{F}$ is given by (\ref{mellin transfrom of F}). Applying Lemma \ref{general lemma}, one obtains $\delta{\rm log}{\rm det}\Delta_{D}=F(\infty)-F(0)$, where $F(\infty)=0$ due to (\ref{asymptotics of Phi plus Psi}). Taking into account (\ref{Psi big 1}), (\ref{Psi 0}), (\ref{Phi at 0}), one arrives at
\begin{equation}
\label{regularized answer}
\delta{\rm log}{\rm det}\Delta_{D}=\frac{1}{6\pi}\Im\Big(\int\limits_{[q,b]}\{z,x\}(A\cdot\nu)\hat{\nu}(x)dx-\frac{e^{i\alpha_i}}{{\rm sin}\alpha_i}\int\limits_{\mathcal{C}}\{z,x\}dx\Big).
\end{equation}
Note that the right-hand side of (\ref{regularized answer}) is finite and  is independent of the choice of the contour $\mathcal{C}\subset P$ (starting at arbitrary $q'\in[x_{i-1},x_i=a]$ and ending at arbitrary $q\in [a,b]$ sufficiently close to $a$). Therefore, one can assume that $\mathcal{C}=\mathcal{C}(\varepsilon)$ is the image $\mathcal{C}(\varepsilon)=x(\mathcal{A}(\varepsilon))$ of the arc $\mathcal{A}(\varepsilon)$ of radius $z(q)-z_i=:\varepsilon$ in the upper half-plane with center at the accessory parameter $z_i$ of $a$. 

Now, passing to the limit as $\varepsilon\to +0$ in (\ref{regularized answer}) is equivalent to the Hadamard regularization (\ref{Main formula}) of the integral $\int_{[a,b]}\{z,x\}(A\cdot\nu)\hat{\nu}(x)dx$ if the expansion of the integral $\int_{\mathcal{C}(\varepsilon)}\{z,x\}dx$ in (\ref{regularized answer}) in powers of $\varepsilon$ contains no constants terms. To show that this is true, we calculate the Schwarzian derivative of $z(x)$ (where $x(z)$ is given by (\ref{Shwarz})) and note that
\begin{align}
\label{Schwarzian derivatives}
\begin{split}
\frac{\partial x}{\partial z}=C&\prod_{k=1}^n(z-z_k)^{\frac{\alpha_k}{\pi}-1}\sim \Big[C\prod_{k\ne i}^n(z_i-z_k)^{\frac{\alpha_k}{\pi}-1}\Big](z-z_i)^{\frac{\alpha_i}{\pi}-1}(1+O(z-z_i)),\\
\{x,z\}=&\sum_{k=1}^n\frac{\pi-\alpha_k}{\pi(z-z_k)^2}-\frac{1}{2}\Big(\sum_{k=1}^n\frac{\pi-\alpha_k}{\pi(z-z_k)}\Big)^2\sim \Big[\frac{\pi^2-\alpha_i^2}{2\pi^2}\Big]\frac{1+O(z-z_i)}{(z-z_i)^{2}},\\
\{z,x\}dx&=-\frac{\partial z}{\partial x}\{x,z\}dz\sim\\
&\sim \Big[\frac{\alpha_i^2-\pi^2}{2C\pi^2}\prod_{k\ne i}^n(z_i-z_k)^{1-\frac{\alpha_k}{\pi}}\Big](z-z_i)^{-\frac{\alpha_i}{\pi}-1}dz(1+O(z-z_i)),\\
x\{z,x\}d&x\sim \frac{dz}{2}\Big(\frac{\alpha_i}{\pi}-\frac{\pi}{\alpha_i}\Big)\frac{1+O(z-z_i)}{z-z_i}
\end{split}
\end{align}
as $z\to z_i=z(a)$, whence
\[\int_{\mathcal{C}(\varepsilon)}\{z,x\}dx\sim {\rm const}\varepsilon^{-\frac{\alpha_i}{\pi}}+o(1)\]
due to the assumption $\alpha_i<\pi$. In view of this observation, formula (\ref{Main formula}) follows from (\ref{regularized answer}).

\section{Infinitesimal rotations of sides near vertices} 
\label{sec rot}
\subsection*{Variation of $\zeta_{\Delta_D-\mu}(2+Q)$.} Now, we aim to calculate the variation $\delta{\rm logdet}\Delta_D$ under the infinitesimal rotation $\delta\alpha_i=1$ of the side $[a,b]$ of $P$ near one of its end (say, $a$). Thus, we consider smooth perturbations of $\partial P$ satisfying
\begin{equation}
\label{odd perturbations}
{\rm supp}(A\cdot\nu)\subset [a,b), \qquad (A\cdot\nu)=x-a \text{ on } [a,q].
\end{equation}
Such deformation of $P$ can be implemented by the use of the family of diffeomorphisms
$$\phi_t:\,x\mapsto x-\upsilon t\,\Im\big((x-a)\overline{\nu}\big)\,\chi(x-a) \qquad \Big(t\in(-t_0,t_0), \quad \upsilon:=\frac{a-x_{i-1}}{|a-x_{i-1}|}\Big),$$
where $\chi$ is a smooth positive cut-off function equal to 1 near the origin and to 0 outside a small neighborhood of the origin.  

In this case, one cannot apply the trick of the Section \ref{sec shift} allowing to shift the contour of integration in (\ref{Hadamard}) away from the vertices since the corresponding differential form $x\omega_j$ is not closed. Thus, it remains to apply a more straightforward and technical method based on the appropriate regularizations of expressions in (\ref{variation of zeta 2 mu}), (\ref{variation of zeta 2 mu short}) by using the Sommerfeld-type formulas for the heat kernel in a cone \cite{Carslaw,Dowker}.

{ First, let us show that formula (\ref{variation of zeta 2 mu short}) with sufficiently large $Q$ is valid for such type of variations. Indeed, formula (\ref{variation of -2 term}) and the well-known Rellich identity
$$\int_{\partial P}\Re[\overline{(x-x_*)}\hat{\nu}(x)]|\partial_\nu u_j|^2dl=2\lambda_j^2$$
(for the case of the Lipschitz domain, see, e.g., Sections 1 and 2, \cite{Ve}) with $x_*\in P$ imply that $\delta(\lambda_j-\mu)^{-(2+Q)}=O(\lambda_j^{-1-Q})$} and the series $\sum_j\delta(\lambda_j-\mu)^{-(2+Q)}$ converge uniformly in $t$ for $Q\ge 3$. This means that the series $\zeta_{\Delta_D-\mu}(2+Q)=\sum_j(\lambda_j-\mu)^{-(2+Q)}$ admit the term-wise differentiation in $t$ for $Q=3,4,\dots$. For such $Q$, formula (\ref{variation of zeta 2 mu short}) is valid and takes the form
\begin{align}
\label{rotation short}
(1+Q)!\delta\zeta_{\Delta_D-\mu}(2+Q)=\int\limits_a^q \partial_\mu^{2+Q}\psi_\mu(x) (x-a)dx+\partial_\mu^{2+Q}\Psi_\mu([q,b]),
\end{align}
where $\psi_\mu$ is defined in (\ref{Psi}) and $\Psi_\mu=\Psi_\mu([q,b])$ is obtained by replacing the integration contour in (\ref{Psi big 1}) by $[q,b]$.

\subsection*{Regilarization of the integral $\int_a^q \partial_\mu\psi_\mu(x) (x-a)dx$.} As shown below, one cannot change the integration and the differentiation in $\mu$ in the right-side of (\ref{rotation short}) since the resulting integral $\int_a^q \partial_\mu\psi_\mu(x) (x-a)dx$ diverges (even for finite $\mu$). To avoid this difficulty, one can add the counter-term of the form
$$\frac{\mathfrak{C}}{x-a}\partial_x K_0(\sqrt{-\mu}(x-a))$$
to the integrand, where $\mathfrak{C}$ is the constant to be determined. In other words, one rewrites (\ref{rotation short}) as
\begin{align}
\label{rotation counter terms}
(1+Q)!\delta\zeta_{\Delta_D-\mu}(2+Q)=\mathfrak{C}\partial_\mu^{1+Q}\Upsilon'_\mu(a,q)+\partial_\mu^{2+Q}\tilde{\Upsilon}_\mu(a,q)+\partial_\mu^{2+Q}\Psi_\mu([q,b]),
\end{align}
where
\begin{equation}
\label{the main part}
\Upsilon'_\mu(a,q):=\int\limits_{a}^{q}\partial_x\partial_\mu[K_0(\sqrt{-\mu}(x-a))]dx
\end{equation}
and
\begin{equation}
\label{the regularized remainder}
\tilde{\Upsilon}_\mu(a',q):=\int\limits_{a'}^{q}\Big[\psi_\mu(x)(x-a)-\mathfrak{C}\partial_x K_0(\sqrt{-\mu}(x-a))\Big]dx \qquad {(a'\in[a,q))}.
\end{equation}
Our aim is to show that there is the unique choice of $\mathfrak{C}$ making integral (\ref{the regularized remainder}) convergent. To this end, we need the asymptotics of the resolvent kernel $R_\mu(x,y)$ for $x,y$ close to $a$.

\subsection*{Parametrix for the resolvent kernel $R_\mu(x,y)$ near the vertex.}
Let $(\varrho,\varphi)$, $(\varrho',\varphi)$ be the polar coordinates of $x,y\in\mathbb{K}$, respectively, { obeying $\varphi(x)=\varphi(y)=0$ for $x,y\in[a,2q]$.} The heat kernel of the Laplacian $\Delta_{\mathbb{K}}$ in $\mathbb{K}$ (the integral kernel of $e^{-t\Delta_{\mathbb{K}}}$) is given by the Sommerfeld integral \cite{Carslaw,Dowker}
\begin{equation}
\label{heat kernel cover}
\mathcal{H}_t(x,y\,|\,\mathbb{K})=\frac{1}{8\pi i\beta t}\int\limits_{\mathscr{E}_+\cup(-\mathscr{E}_+)}\Bigg({\rm exp}\Big(-\frac{\mathfrak{R}(\varrho,\varrho',\vartheta)}{4t}\Big)-1\Bigg)\,\Xi(\vartheta,\varphi-\varphi') d\vartheta,
\end{equation}
where
\[\mathfrak{R}(\varrho,\varrho',\vartheta):=\varrho^2-2\varrho\varrho'{\rm cos}\vartheta+\varrho'^{2}, \qquad \Xi(\vartheta,\theta):={\rm cot}\Big(\pi\beta^{-1}(\vartheta+\theta)\Big)\]
and the contour $\mathscr{E}_+$ is contained in the semi-strip $\Re\vartheta\in[-\pi,\pi]$, $\Im\vartheta>0$ and runs from $\vartheta=\pi+i\infty$ to $\vartheta=-\pi+i\infty$ The contour $\mathscr{E}_+\cup(-\mathscr{E}_+)$ can be replaced by the union of small anticlockwise circles with centered at the poles of $\Xi$ lying in the small neighborhood of $\vartheta=0$ and the two contours $\mathscr{E}_L$, $\mathscr{E}_R$ running from $-\pi-i\infty$ to $-\pi+i\infty$ and from $\pi-i\infty$ to $\pi+i\infty$, respectively, and avoiding the poles of $\Xi$ and the zone $\Re {\rm sin}^2(\vartheta/2)\le 0$.

The resolvent kernel of $\Delta_{\mathbb{K}}$ is related to its heat kernel via the Laplace transform 
\begin{equation}
\label{resolvent via heat}
R_{\mu}(x,y\,|\,\mathbb{K})=\int\limits_0^\infty e^{\mu t}\mathcal{H}_t(x,y\,|\,\mathbb{K})\,dt.
\end{equation}
Let $\chi$ be a smooth cut-off function on $X$ equal to 1 in the $2q$-neighborhood of $a$ and to 0 outside $3q$-neighborhood of $a$. For $\varrho'\le q$, let us represent the resolvent kernel as 
\begin{equation}
\label{resolvent via heat asymp}
R_\mu(x,y)={\chi(x)}R_{\mu}(x,y\,|\,\mathbb{K})+\check{R}_\mu(x,y);
\end{equation}
then the discrepancy $(\Delta-\mu)\check{R}_\mu(\cdot,y)$ is supported on the set $\{x\in X \,|\,\varrho\in (2q,3q)\}$ and it and all its partial derivatives with respect to $x,y$ decay as $O(\mu^{-\infty})$ as $\Re\mu\to-\infty$. Due to the standard estimate $(\Delta-\mu)^{-1}=O(\mu^{-1})$ of the resolvent, the smoothness increasing theorems for solutions to elliptic equations and the Sobolev embedding theorem one has
\begin{align*}
\|D^{l_1,l_2}_y\check{R}_\mu&(\cdot,y)\|_{C^{2l}(U;\mathbb{R})}\le c\|D^{l_1,l_2}_y\check{R}_\mu(\cdot,y)\|_{H^{2(l+1)}(U;\mathbb{R})}\le \\
\le &c(\|D_y\Delta^{l+1}\check{R}_\mu(\cdot,y)\|_{L_2(X;\mathbb{R})}+\|D^{l_1,l_2}_y\check{R}_\mu(\cdot,y)\|_{L_2(X;\mathbb{R})})=O(\mu^{-\infty}),
\end{align*}
uniformly in $y$, where $l,l_1,l_2=0,1,\dots$, $U$ is an arbitrary domain in $X$ whose closure does not contain vertices and $D^{l_1,l_2}_y=\partial_y^{l_1}\partial_{\overline{y}}^{l_2}$. 

{ To estimate near the vertex $a$, recall the expansion 
\begin{equation}
\label{solutions in cone}
u(x)=\frac{1}{\beta \tilde{q}}\sum_{k\in\mathbb{Z}}\frac{J_{2\pi|k|/\beta}(\varrho\sqrt{\mu})}{J_{2\pi|k|/\beta}(\tilde{q}\sqrt{\mu})}e^{2\pi k i\varphi/\beta}(u,e^{2\pi k i\varphi/\beta})_{L_2(\{x\in X\,|\, \varrho(x)=\tilde{q}\})} \qquad (\varrho(x)\le \tilde{q}),
\end{equation}
which is valid for any function $u$ which is bounded and obeys $(\Delta-\mu)u=0$ ($\mu>0$) in the $2q$-neighborhood of $a$ in $X$; here $J_\nu$ is the Bessel function. Indeed, each term in the sum in (\ref{solutions in cone}) obeys the equation $(\Delta-\mu)v=0$ while (\ref{solutions in cone}) becomes the Fourier expansion of the restriction $u$ on the arc $\{x\in X\,|\, \varrho(x)=\tilde{q}\}$ if $\varrho=\tilde{q}$. Note that $\tilde{q}\in(q,2q)$ can always be chosen in such a way that $\mu$ is far from the zeros of 
$$J_{2\pi|k|/\beta}(\tilde{q}\sqrt{\mu})\underset{\mu\to +\infty}{\simeq}\sqrt{\frac{2\pi}{x}}{\rm cos}\Big(\tilde{q}\sqrt{\mu}-\frac{\pi^2 k}{\beta}-\frac{\pi}{4}\Big)$$ 
(which are exactly the eigenvalues of the Dirichlet Laplacian on the finite cone $\{x\in X \ | \ r\le \tilde{q}\}$).} Since $\check{R}_\mu(\cdot,y)$ obeys $(\Delta-\mu)\check{R}_\mu(\cdot,y)=0$ near the vertex $a$, formula (\ref{solutions in cone}) yields
\begin{equation}
\label{expansion for check R}
\check{R}_\mu(x,y)=\frac{1}{\beta q}\sum_{k\in\mathbb{Z}}\frac{J_{2\pi|k|/\beta}(\varrho\sqrt{\mu})}{J_{2\pi|k|/\beta}(q\sqrt{\mu})}e^{2\pi k i\varphi/\beta}(\check{R}_\mu(\cdot,y),e^{2\pi k i\varphi/\beta})_{L_2(\{x\in X\,|\,\varrho=q\})}.
\end{equation}
In view of this expansion, the above estimates imply that 
\begin{equation}
\label{parametrix remainder esti}
\check{R}_\mu(x,y)=O(\mu^{-\infty}), \qquad {\varrho^{-2\pi/\beta}\varrho^{'-2\pi/\beta}
\partial_{\varphi}\partial_{\varphi'}\check{R}_\mu(x,y)=O(\mu^{-\infty}),}
\end{equation}
where $O(\mu^{-\infty})$ denotes the remainder uniformly bounded for all $x,y,\mu$ obeying $\varrho,\varrho'\le q$ and $\Re\mu<\mu_0$, $|\mu|>\mu_0$ ($\mu_0>0$ can be arbitrarily small) and decaying faster than any power of $\mu$ and uniformly in $x,y$ obeying $\varrho,\varrho'\le q$ as $\mu\to -\infty$. 

\subsection*{Completing the regilarization of the integral $\int_a^q \partial_\mu\psi_\mu(x) (x-a)dx$.} { Taking into account formulas (\ref{parametrix remainder esti}), (\ref{Psy asymp infinity}), (\ref{Psi}), one can rewrite (\ref{the regularized remainder}) as
\begin{align}
\label{mathfrak j}
\begin{split}
\mathfrak{J}_{\mu,\epsilon}=&\tilde{\Upsilon}_\mu(a+\epsilon,q)-[1+{\rm log}2-(1/2){\rm log}(-\mu)-\gamma]\frac{\mu}{2\pi}\cdot\frac{(q-a)^2-\epsilon^2}{2}=\\
=&\int\limits_{\epsilon}^{q-a}\Big[\partial_{\varphi'}\partial_{\varphi}[\breve{R}_\mu(x,y)-\breve{R}_\mu(x,\overline{y})]\big|_{\varphi=\varphi'=0}\varrho^{-1}-\mathfrak{C}\partial_\varrho K_0(\sqrt{-\mu}\varrho)\Big]d\varrho,
\end{split}
\end{align}
where $\epsilon\in[0,q-a)$ and $\breve{R}_\mu$ is the remainder in (\ref{resolvent asymptotics}). In view of formulas (\ref{resolvent asymptotics}), (\ref{resolvent via heat asymp}), (\ref{resolvent via heat}), the equality
\begin{equation}
\label{McDonald}
K_0(\sqrt{-\mu}\varrho)=\int_0^{+\infty}{\rm exp}\Big(\mu t-\frac{\varrho^2}{4t}\Big)\frac{dt}{2t},
\end{equation}
and the residue theorem, one has
\begin{align*}
\partial_{\varphi'}\partial_{\varphi}\big[\breve{R}_\mu(x,y)-\check{R}_\mu(x,y)\big]\big|_{\varphi=\varphi'=0}=\int\limits_0^\infty e^{\mu t}\partial_{\varphi'}\partial_{\varphi}\Big[\mathcal{H}_t(x,y\,|\,\mathbb{K})-\frac{1}{4\pi t}{\rm exp}\Big(-\frac{r^2}{4t}\Big)\Big]\big|_{\varphi=\varphi'=0}dt=\\
=\int\limits_0^\infty e^{\mu t}\frac{1}{8\pi i\beta t}\int\limits_{\mathscr{E}_L\cup\mathscr{E}_R}(e^{-\varkappa}-1)\,\partial_{\varphi'}\partial_{\varphi}\Xi(\vartheta,\varphi-\varphi')\big|_{\varphi=\varphi'=0}d\vartheta\,dt,
\end{align*}
where 
$$\varkappa(\varrho,\vartheta,t):=\frac{\varrho^2{\rm sin}^2(\vartheta/2)}{t}.$$ 
Therefore,
\begin{align*}
\mathfrak{J}_{\mu,\epsilon}=\int\limits_{\epsilon}^{q-a}d\varrho\,\int\limits_0^\infty\,\frac{dt\,e^{\mu t}}{t}\Big[-\frac{\mathfrak{C}}{2}\partial_{\varrho}e^{-\varrho^2/4t}+\frac{1}{8\pi i\beta\varrho}\int\limits_{\mathscr{E}_L\cup\mathscr{E}_R}(e^{-\varkappa}-1)\Big[\partial_{\varphi}\partial_{\varphi'}\sum_\pm \pm\Xi(\vartheta,\varphi\mp\varphi')\Big]_{\varphi',\varphi=0}d\vartheta\Big]+\\
+\int\limits_{\epsilon}^{q-a}\partial_{\varphi'}\partial_{\varphi}\partial_\mu^l[\check{R}_\mu(x,y)-\check{R}_\mu(x,\overline{y})]\big|_{\varphi=\varphi'=0}\varrho^{-1}d\varrho
\end{align*}
Due to (\ref{expansion for check R}) and (\ref{parametrix remainder esti}), the second term in the right-hand side is $O(\mu^{-\infty})$ uniformly in $\epsilon\in[0,q-a)$. Interchanging the integrations in $\varrho$ and $t$ in the first term under condition $\epsilon>0$ or $l>0$ yields the (uniform in $\epsilon$) asymptotics
\begin{align*}
\mathfrak{J}_{\mu,\epsilon}=O(&\mu^{-\infty})+\int\limits_0^{1}\frac{dt\,e^{\mu t}}{t}\times\\
\times&\Bigg[\frac{1}{8\pi i\beta}\int\limits_{\mathscr{E}_L\cup\mathscr{E}_R}d\vartheta\,\Big[\partial_{\varphi}\partial_{\varphi'}\sum_\pm \pm\Xi(\vartheta,\varphi\mp\varphi')\Big]_{\varphi',\varphi=0}\int\limits_{\epsilon}^{q-a}\Big(e^{-\varkappa}-1\Big)\frac{d\varrho}{\varrho}-\frac{\mathfrak{C}}{2}e^{-p}\Bigg|^{(q-a)^2/4t}_{\epsilon^2/4t}\Bigg]=\\
&=O(\mu^{-\infty})+\int\limits_0^{1}\frac{dt\,e^{\mu t}}{t}\Bigg[\frac{\mathfrak{C}}{2}e^{-\epsilon^2/4t}-\frac{1}{8\pi i\beta}\int\limits_{\mathscr{E}_L\cup\mathscr{E}_R}\int\limits_{\varkappa(\epsilon,\vartheta,t)}^{\varkappa(q-a,\vartheta,t)}\Big(e^{-\varkappa}-1\Big)\frac{d\varkappa}{\varkappa}d\Xi'(\vartheta,0)\Bigg].
\end{align*}
Using the representation
$$\int\limits_{h}^{p}\frac{e^{-\varkappa}-1}{\varkappa}d\varkappa={\rm log}(h/p)+E_1(h)-E_1(p)$$
(where $E_1(z):=\int_z^{\infty}\frac{e^{-\varkappa}}{\varkappa}d\varkappa$ is the exponential integral), one rewrites the latter integral as
\begin{align*}
\int\limits_0^{1}\frac{dt\,e^{\mu t}}{t}\Bigg[\frac{\mathfrak{C}}{2}e^{-\epsilon^2/4t}+\frac{1}{8\pi i\beta}\int\limits_{\mathscr{E}_L\cup\mathscr{E}_R}\Big[E_1(\varkappa(\epsilon,\vartheta,t))-E_1(\varkappa(q,\vartheta,t))\Big]d\Xi'(\vartheta,0)\Bigg]=\\
=\int\limits_0^{\epsilon^{-2}}du\,u^{-1}e^{\mu\epsilon^2u}\Bigg[\frac{\mathfrak{C}}{2}e^{-1/4u}-\frac{1}{8\pi i\beta}\int\limits_{\mathscr{E}_L\cup\mathscr{E}_R}E_1\Big(\frac{{\rm sin}^2(\vartheta/2)}{u}\Big)d\Xi'(\vartheta,0)\Bigg]+\\
+\int\limits_0^{1}\frac{dt\,e^{\mu t}}{t}\Bigg[\frac{1}{8\pi i\beta}\int\limits_{\mathscr{E}_L\cup\mathscr{E}_R}E_1(\varkappa(q,\vartheta,t))d\Xi'(\vartheta,0)\Bigg]
\end{align*}
Since $\varkappa(q,\vartheta,t)>c>0$ on $\mathscr{E}_L\cup\mathscr{E}_R$, the last integral converges and is equal to $O(\mu^{-\infty})$. So, we have
\begin{align*}
\mathfrak{J}_{\mu,\epsilon}=O(\mu^{-\infty})+\int\limits_0^{\epsilon^{-2}}du\,u^{-1}e^{\mu\epsilon^2u}\Bigg[\frac{\mathfrak{C}}{2}e^{-1/4u}-\frac{1}{8\pi i\beta}\int\limits_{\mathscr{E}_L\cup\mathscr{E}_R}E_1\Big(\frac{{\rm sin}^2(\vartheta/2)}{u}\Big)d\Xi'(\vartheta,0)\Bigg].
\end{align*}
Here the integrand rapidly decays as $u\to+0$. From the expansion
$$E_1(z)=-\gamma-{\rm log}(z)+z+\dots$$
near $z=0$, it follows that the expression in the in square brackets is equal to
$$\frac{\mathfrak{C}}{2}+\frac{1}{8\pi i\beta}\int\limits_{\mathscr{E}_L\cup\mathscr{E}_R}{\rm log}\Big({\rm sin}^2(\vartheta/2)\Big)d\Xi'(\vartheta,0)+O(u^{-\delta})$$
for large $u$; here $\delta\in(0,1)$. Therefore, the integral $\mathfrak{J}_{\mu,0}$ converges if and only if
\begin{align}
\label{corner contribution}
\mathfrak{C}=\mathfrak{C}(\beta):=-\frac{1}{4\pi i\beta}\int\limits_{\mathscr{E}_L\cup\mathscr{E}_R}{\rm log}\Big({\rm sin}^2(\vartheta/2)\Big)d\Xi'(\vartheta,0)=-\frac{2\pi}{8i\pi\beta^2}\int\limits_{\mathscr{E}_L\cup\mathscr{E}_R}
\frac{{\rm cot}(\vartheta/2)}{{\rm sin}^2(\pi\beta^{-1}\vartheta)}d\vartheta=\frac{1}{6\beta}\Big(\frac{2\pi}{\beta}-\frac{\beta}{2\pi}\Big).
\end{align}
If (\ref{corner contribution}) holds, then 
\begin{equation}
\label{mathfrak j asymp}
\mathfrak{J}_{\mu,0}=\mathfrak{J}_{\infty,0}(\beta)+O(\mu^{-\infty}),
\end{equation}
where
\begin{align*}
\mathfrak{J}_{\infty,0}(\beta)=-\frac{1}{8\pi i\beta}\int\limits_{\mathscr{E}_L\cup\mathscr{E}_R}\mathfrak{H}({\rm sin}^2(\vartheta/2))d\Xi'(\vartheta,0),
\end{align*}
and
$$\mathfrak{H}(\sigma):=\int\limits_0^{+\infty}\frac{du}{u}\Big[E_1\Big(\frac{\sigma}{u}\Big)-E_1\Big(\frac{1}{u}\Big)+e^{-1/4u}{\rm log}(\sigma)\Big].$$
By definition, $\mathfrak{H}(1)=0$. Differentiating in $\sigma$, one obtains
$$\mathfrak{H}'(\sigma):=\int\limits_0^{+\infty}\frac{du}{\sigma u}\Big[e^{-1/4u}-e^{-\sigma/u}\Big]=\int\limits_0^{+\infty}\frac{dv}{v\sigma}\Big[e^{-v/4}-e^{-\sigma v}\Big]=\frac{{\rm log}(4\sigma)}{\sigma}.$$
Hence $\mathfrak{H}(\sigma)=\frac{1}{2}{\rm log}^2(\sigma)+2{\rm log}(\sigma)\,{\rm log}(2)$ and
\begin{equation}
\label{mathfrak j main term}
\begin{split}
\mathfrak{J}_{\infty,0}(\beta)=\mathfrak{C}(\beta){\rm log}(2)-\mathfrak{K}(\beta)=\\
\end{split}
\end{equation}
where the integral
\begin{equation}
\label{mathfrak K}
\begin{split}
\mathfrak{K}(\beta)&=\frac{1}{16\pi i\beta}\int\limits_{\mathscr{E}_L\cup\mathscr{E}_R}{\rm log}^2({\rm sin}^2(\vartheta/2))d\Xi'(\vartheta,0)=\\
=&\frac{\pi}{2\beta}\lim_{\epsilon\to+0}\Bigg(\int_{\epsilon}^{+\infty}\frac{dy}{{\rm tanh}(y/2){\rm sinh}^2(\pi y/\beta)}-\frac{\beta}{\pi^2\epsilon^2}-\Big(\frac{2}{3}-\frac{\beta^2}{6\pi^2}\Big){\rm log}(\epsilon/2)+\frac{\beta^2}{12\pi^2}\Bigg)=\mathscr{K}(\beta/2)
\end{split}
\end{equation}
is calculated in the same way as in Subsection 4.2, \cite{KKsphere} and $\mathscr{K}$ is given by (\ref{mathscr K}). The similar reasoning shows that $\partial_\mu\mathfrak{J}_{\mu,0}=O(\mu^{-\infty})$, i.e. asymptotics (\ref{mathfrak j asymp}) admits differentiation.} 

In addition, formulas (\ref{the main part}) and (\ref{McDonald}) yield
\begin{align*}
\Upsilon'_\mu(a,q):=&\frac{1}{2}\int\limits_0^{+\infty}dt\int\limits_0^{q-a}\partial_\varrho{\rm exp}\Big(\mu t-\frac{\varrho^2}{4t}\Big)d\varrho=\\
=&\frac{1}{2}\int\limits_0^{+\infty}dt\,e^{\mu t}\Big({\rm exp}\Big(-\frac{(q-a)^2}{4t}\Big)-1\Big)=\partial_\mu\big[{\rm log}\sqrt{-\mu}+K_0(\sqrt{-\mu}(q-a))\big].
\end{align*}

\subsection*{Completing the derivation of $\delta\zeta_{\Delta_D-\mu}(2+Q)$.} As a result, formula  (\ref{rotation counter terms}) can be rewritten as 
\begin{equation}
\label{rotation counter terms 1}
(1+Q)!\delta\zeta_{\Delta_D-\mu}(2+Q)=\partial_{\mu}^{2+Q}F(\mu), \qquad F(\mu)=\mathfrak{C}(\beta)\Big[{\rm log}\sqrt{-\mu}+K_0(\sqrt{-\mu}(q-a))\Big]+\tilde{\Upsilon}_\mu(a,q)+\Psi_\mu([q,b]).
\end{equation}
{ In view of (\ref{mathfrak j}), (\ref{corner contribution}), and (\ref{Psy asymp infinity}), (\ref{mathfrak j asymp}), the asymptotics
\begin{equation}
\label{high mu asymp rotation 0}
F(\mu)=\mathfrak{C}(\beta){\rm log}\sqrt{-\mu}+\delta|P|\cdot[1+{\rm log}2-(1/2){\rm log}(-\mu)-\gamma]\mu/2\pi+\mathfrak{J}_{\infty,0}(\beta)+O(\mu^{-\infty})
\end{equation}
is valid, where $\mathfrak{C}(\beta)$ and $\mathfrak{J}_{\infty,0}(\beta)$ are given by (\ref{corner contribution}) and (\ref{mathfrak j main term});} 
in particular
\begin{equation}
\label{high mu asymp rotation}
F''(\mu)=\frac{\delta|P|}{4\pi(-\mu)}-\frac{\mathfrak{C}(\beta)}{2\mu^2}+O(\mu^{-\infty})
\end{equation}
{ and
\begin{equation}
\label{F infty is not zero!}
F(\infty)=\mathfrak{J}_{\infty,0}(\beta).
\end{equation}
Let} us calculate $F(0)$. We have
\begin{equation}
\label{rotatio F 0}
F(0)=\mathfrak{C}(\beta)\big({\rm log}2-\gamma\big)+\tilde{\Upsilon}_0(a,q)+\big[\Psi_0([q,b])-\mathfrak{C}(\beta){\rm log}(q-a)\big].
\end{equation}
From (\ref{the regularized remainder}), (\ref{Psi 0}) and the asymptotics $-\partial_z K_0(z)=z^{-1}+o(1)$ as $z\to 0$, it follows that 
$$\tilde{\Upsilon}_\mu(a,q):=\int\limits_{a}^{q-a}\Bigg[\frac{\mathfrak{C}(\beta)}{x-a}-\frac{(x-a)\Re\{z,x\}}{6\pi}\Bigg]dx=o(1) \qquad (q\to a).$$
Here the last identity follows from (\ref{Schwarzian derivatives}). Now the limit transition in (\ref{rotatio F 0}) yields
\begin{equation}
\label{rotatio F 0 fin}
\begin{split}
F(0)=\mathfrak{C}(\beta)&\big({\rm log}2-\gamma\big)+\lim_{q\to+0}\big[\Psi_0([q,b])-\mathfrak{C}(\beta){\rm log}(q-a)\big]=\\
=&\mathfrak{C}(\beta)({\rm log}2-\gamma)-\frac{1}{6\pi}\Im\mathscr{H}\int\limits_{\partial P}\{z,x\}(A\cdot\nu)\hat{\nu}(x)dx.
\end{split}
\end{equation}
\subsection*{Variation of $\zeta_{\Delta_D}(s)$: formal calculations.} Formally differentiating (\ref{zeta s via zeta 2 mu}) in $t$, applying (\ref{rotation counter terms 1}), (\ref{high mu asymp rotation}) integrating by parts, one arrives at $(s-1)\delta\zeta_{\Delta_D}(s)=\widehat{F}(s)$, where $\widehat{F}$ is given by (\ref{mellin transfrom of F}). Now Lemma \ref{general lemma} yields $\delta{\rm log}{\rm det}\Delta_{D}=F(\infty)-F(0)$, where { $F(\infty)$ and $F(0)$ are given by (\ref{rotatio F 0 fin}) and (\ref{F infty is not zero!})}. The last equality is exactly formula (\ref{Main formula}) for variation (\ref{odd perturbations}).

\section{Justifying formula (\ref{Main formula}) for polygon-preserving deformations}
\label{justy}
Now we give the rigorous proof of (\ref{Main formula}) on the class of polygon-preserving deformations $t\mapsto P_t$ ($t\in[-t_0,t_0]$). By additivity, it is sufficient to consider the linear deformations of the one side (as the corresponding family $P_t$, one can take $P_t=\phi_t(P_0)$, where $\phi_t$ is of the form (\ref{shif plus rot})). Then one can assume that $(A\cdot\nu)(x)=\mathcal{A}+\mathcal{B}x$ on $[x_i,x_{i+1}]=[a(t),b(t)]\subset\mathbb{R}$ and $(A\cdot\nu)=0$ on $\partial P\backslash [a,b]$. Let $\chi_a$, $\chi_b$ be smooth non-negative cut-off functions obeying $\chi_a=1$ near $a$, $\chi_b=1$ near $b$ and $\chi_a+\chi_b=1$ on $[a,b]$. Then $(A\cdot\nu)$ is the sum of the following infinitesimal deformations
$$\chi_a(x)[\mathcal{A}+\mathcal{B}a], \quad \chi_b(x)[\mathcal{A}+\mathcal{B}b], \quad \chi_a(x)\mathcal{B}\cdot(x-a), \quad -\chi_b(x)\mathcal{B}\cdot(b-x)$$
considered in Sections \ref{sec shift} and \ref{sec rot}. Note that the expressions (\ref{variation of zeta 2 mu shift 1}), (\ref{rotation counter terms 1}) for $\delta\zeta_{\Delta_D-\mu}(2+Q)$ (with sufficiently large $Q$) are justified for each of the above deformations. This, by additivity, $\zeta_{\Delta_D-\mu}(2+Q)$ is differentiable in $t$ for all $t\in[-t_0,t_0]$ and sufficiently large $Q=3,\dots$, and 
$$(1+Q)!\delta\zeta_{\Delta_D-\mu}(2+Q)=\partial^{2+Q}_\mu F(\mu)$$
where
\begin{align}
\label{F gen case}
\begin{split}
F(\mu)=\Psi_\mu([q_1,q_2])+[\mathcal{A}+\mathcal{B}a]\Phi_\mu(a,q_1)+[\mathcal{A}+\mathcal{B}b]\Phi_\mu(b,q_2)+\\
+\mathcal{B}\Big(\mathfrak{C}(2\alpha_i)\Big[{\rm log}\sqrt{-\mu}+K_0(\sqrt{-\mu}(q_1-a))\Big]+\tilde{\Upsilon}_\mu(a,q)\Big)-\\
-\mathcal{B}\Big(\mathfrak{C}(2\alpha_{i+1})\Big[{\rm log}\sqrt{-\mu}+K_0(\sqrt{-\mu}(b-q_2))\Big]+\tilde{\Upsilon}_\mu(b,q_2)\Big),
\end{split}
\end{align}
where $\Psi_\mu([q_1,q_2])$ is obtained by replacing the integration contour in (\ref{Psi big 1}) by $[q,b]$. Similarly, $\Phi_\mu(b,q_2)$ is obtained by replacing $\alpha_i$ by $\alpha_{i+1}$ and $\mathcal{C}$ by the curve starting at $q_2$ and ending at $q'_2\in[x_{i+1},x_{i+2}]$ in (\ref{Phi 01}), while $\tilde{\Upsilon}_\mu(b,q_2)$ is obtained by replacing $x-a$ by $q-x$ and the interval $[a,q]$ with $[q_2,b]$ in (\ref{the regularized remainder}).

In view of asymptotics (\ref{Psi big 2 as}), (\ref{asymptotics of Phi}), and (\ref{high mu asymp rotation 0}), (\ref{high mu asymp rotation}), formulas (\ref{area var}), (\ref{perimeter variation}), (\ref{corner contribution}), { (\ref{mathfrak j main term}), (\ref{mathfrak K}),} and the obvious identity $\delta\big(\sum_i\alpha_i\big)=0$, formula (\ref{F gen case}) yields
\begin{equation}
\label{asumptotics large mu general}
{ \begin{split}
F(\mu)&=(-\mu){\rm log}(-\mu)\frac{\delta|P|}{4\pi}+\frac{\delta{\bm|}\partial P{\bm|}}{4}\sqrt{-\mu}-2\delta{\bm\beta}_1(P)\cdot{\rm log}\sqrt{-\mu}+\\
+&\Big[(1+{\rm log}2-\gamma)\delta|P|+(q_1+q_2-a-b)\mathcal{B}\Big]\frac{\mu}{2\pi} -2\delta{\bm\beta}_1(P){\rm log}(2)-\sum_{i=1}^n\mathscr{K}(\alpha_i)\delta\alpha_i+O(\mu^{-\infty}) \qquad (\mu\to-\infty),
\end{split}}
\end{equation}
where
$${\bm\beta}_1(P)=\sum_{i=1}^n\frac{\pi^2-\alpha_i^2}{24\pi\alpha_i}$$
{and $\mathscr{K}$ is given by (\ref{mathscr K}).} In particular,
\begin{equation}
\label{asumptotics large mu general 1}
\partial_{\mu}^{2+Q}F(\mu)=\delta\partial_\mu^Q\mathfrak{F}_{sing}(\mu)+O(\mu^{-\infty}) \qquad (\mu\to-\infty), 
\end{equation} 
for any $Q$, where
\begin{equation}
\label{singular part in mu rep}
\mathfrak{F}_{sing}(\mu):=\frac{|P|}{4\pi(-\mu)}-\frac{{\bm|}\partial P{\bm|}}{8(-\mu)^{3/2}}+\frac{{\bm\beta}_1(P)}{\mu^2}
\end{equation}
Here the crucial moment is that, since the deformations $t\mapsto P_t$ preserve the class of polygons, the asymptotics (\ref{asumptotics large mu general}), (\ref{asumptotics large mu general 1}) are {\it valid and unifrom} for all $t\in[-t_0,t_0]$.

At the same time, according to the results of \cite{Kac,Mckean,Berg}, the heat trace $K(\tau|\Delta_{D}):={\rm Tr}\,e^{-\tau\Delta_{D}}=\sum_k e^{-\lambda_k \tau}$ of $\Delta_D$ admits the asymptotics 
\begin{equation}
\label{heat trace asymp}
K(\tau|\Delta_{D})=\frac{|P|}{4\pi \tau}-\frac{|\partial P|}{8\sqrt{\pi\tau}}+{\bm\beta}_1(P)+O(e^{-{\bm\beta}_2(P)/\tau})\qquad \tau\to +0,
\end{equation}
where ${\bm\beta}_2(P)>0$. Since the zeta function of $\Delta_D-\mu$ is related to the heat trace via
$$\zeta_{\Delta_{D}-\mu}(s)=\frac{1}{\Gamma(s)}\int\limits_{0}^{+\infty}\tau^{s-1}e^{\mu\tau}K(\tau|\Delta_{D})d\tau \qquad (\Re s>1),$$
the decomposition
\begin{equation}
\label{excluding singular part 0}
(1+Q)!\zeta_{\Delta^{(t)}_{D}-\mu}(2+Q)-\partial_\mu^Q\mathfrak{F}_{sing}(\mu)=:\partial_\mu^Q\mathfrak{Z}(\mu)=O(\mu^{-\infty}) \qquad (\mu\to-\infty)
\end{equation}
follows from (\ref{heat trace asymp}) and (\ref{singular part in mu rep}). Thus, formula (\ref{zeta s via zeta 2 mu}) can be rewritten as
\begin{equation}
\label{zeta s decomposition}
(s-1)\dots(s-1-Q)\zeta_{\Delta}(s)-\frac{1}{2\pi i}\int\limits_\Gamma\partial_\mu^Q\mathfrak{F}_{sing}(\mu)\mu^{1+Q-s}d\mu=\frac{1}{2\pi i}\int\limits_\Gamma\partial_\mu^Q\mathfrak{Z}(\mu)\mu^{1+Q-s}d\mu\qquad (\Re s>1).
\end{equation} 
Due to (\ref{excluding singular part 0}), the integral in the right-hand side of (\ref{zeta s decomposition}) is well-defined for any $s\in\mathbb{C}$, {\it no analytic continuation in needed here}. Thus, formula (\ref{zeta s decomposition}) is valid (for the analytic continuation of $\zeta_{\Delta}(s)$) for any $s\in\mathbb{C}$. Moreover, since the functions $\partial_\mu^Q\mathfrak{Z}$ and $\delta\partial_\mu^Q\mathfrak{Z}$ decrease as $\Re\mu\to-\infty$ rapidly and {\it uniformly} in $t\in[-t_0,t_0]$ due to (\ref{excluding singular part 0}) and (\ref{asumptotics large mu general 1}), one can apply, for any $s\in\mathbb{C}$, the Leibniz integral rule to the integral in the right-hand side of (\ref{zeta s decomposition}), obtaining
\begin{equation}
\label{zeta s decomposition 1}
\delta\Bigg[(s-1)\dots(s-1-Q)\zeta_{\Delta}(s)-\frac{1}{2\pi i}\int\limits_\Gamma\partial_\mu^Q\mathfrak{F}_{sing}(\mu)\mu^{1+Q-s}d\mu\Bigg]=\frac{1}{2\pi i}\int\limits_\Gamma\delta\partial_\mu^Q\mathfrak{Z}(\mu)\mu^{1+Q-s}d\mu\qquad (s\in\mathbb{C}).
\end{equation} 
Here $\zeta_{\Delta}(s)$ is the meromorphic continuation of zeta function of $\Delta$ to the whole complex plane. Due to explicit formula (\ref{singular part in mu rep}), it is easily checked that the meromorphic continuation of the integral in the left-hand side of (\ref{zeta s decomposition 1}) commutes with it differentiation in $t$,
$$\delta\int\limits_\Gamma\partial_\mu^Q\mathfrak{F}_{sing}(\mu)\mu^{1+Q-s}d\mu=\int\limits_\Gamma\delta\partial_\mu^Q\mathfrak{F}_{sing}(\mu)\mu^{1+Q-s}d\mu$$
Transferring this integral to the right-hand side of (\ref{zeta s decomposition 1}), one justifies (for any $s\in\mathbb{C}$) the formula
\begin{equation}
\label{zeta s via zeta 2 mu just}
(s-1)\dots(s-1-Q)\delta\zeta_{\Delta}(s)={\frac{1}{2\pi i}}\int\limits_\Gamma\partial^{2+Q}_\mu F(\mu)\mu^{1+Q-s}d\mu=\frac{(Q+1)!}{2\pi i}\int\limits_\Gamma\delta\zeta_{\Delta-\mu}(2+Q)\mu^{1+Q-s}d\mu,
\end{equation} 
obtained by formal differentiation of (\ref{zeta s via zeta 2 mu}) in $t$. 
\begin{rem}
The above justification of the commutation of the differentiation with respect to the parameter $t$ and the analytic continuation in $s$ requires the asymptotics of $\delta\zeta_{\Delta_D-\mu}(2+Q)$ which are uniform in $t\in[-t_0,t_0]$. Since we derive such an asymptotics {\rm(}by the use of the double cover $X$ of $P$ and the parametrix of the heat kernel on it{\rm)} only for the polygons, to this end we need to restrict ourselves to considering the polygon-preserving deformations only. This is the reason why we have not completed the justification of {\rm(\ref{Main formula})} for the local deformations considered in Sections \ref{sec bumping}-\ref{sec rot} and did not consider the general case in which the variations $(A,\nu)$ are smooth on the each side of $P$. In principle, this can be done by implementing the above scheme but it needs the near-diagonal and high-frequency asymptotics of the resolvent kernel which are valid if $P$ is curvilinear polygon. 
\end{rem}
Now, integrating by parts in (\ref{zeta s via zeta 2 mu just}) and applying asymptotics (\ref{asumptotics large mu general 1}), one obtains 
$$(s-1)\delta\zeta_{\Delta}(s)=\frac{1}{2\pi i}\int\limits_\Gamma\partial^{2}_\mu F(\mu)\mu^{1-s}d\mu=:\widehat{F}(s).$$
Applying Lemma \ref{general lemma}, one obtains $\delta{\rm log}{\rm det}\Delta_{D}=F(\infty)-F(0)$, where 
$${ F(\infty)=-2\delta{\bm\beta}_1(P){\rm log}(2)-\sum_{i=1}^n\mathscr{K}(\alpha_i)\delta\alpha_i}$$ 
is the constant term in asymptotics (\ref{asumptotics large mu general 1}). In view of (\ref{Psi 0}), (\ref{Phi 01}), (\ref{Schwarzian derivatives}), (\ref{rotatio F 0 fin}), and (\ref{F gen case}), one has
$$F(0)=\frac{{\rm log}2-\gamma}{12}\sum_{i=1}^n\Big(\frac{\pi}{\alpha_i}-\frac{\alpha_i}{\pi}\Big)\,\frac{\delta\alpha_i}{\alpha_i}-\frac{1}{6\pi}\Im\mathscr{H}\int\limits_{\partial P}\{z,x\}(A\cdot\nu)\hat{\nu}(x)dx.$$
Thus, formula (\ref{Main formula}) is proved.

\section{Appendix: on Wiegmann-Zabrodin formula} 
Let $\Omega$ be a simply connected domain in ${\mathbb C}$ with boundary $\Gamma$;  we assume that
the Riemann biholomorphic map $R_\Omega: w\mapsto z(w)$ from the open unit disk $U=\{|w|<1\}$ to $\Omega$ has a biholomorphic extension to a vicinity of $S^1=\partial U$, in particular the boundary $\Gamma$ is  analytic.

Let $\Delta_\Omega$ be the operator of the Dirichlet boundary value problem in $\Omega$ for the (positive) Laplacian $-4\partial_z\partial_{\bar z}$ and let
${\rm det}\, \Delta_\Omega$ be the $\zeta$-regularized determinant of this operator. 
Let $V$ be a function holomorphic in a vicinity of the unit disk $U$; define a perturbation, $\Omega_\epsilon$, of the domain $\Omega$ as
\[\Omega_\epsilon=(R_\Omega+\epsilon  V)(U).\] 
with small (real) $\epsilon$. Note that the map $R_\Omega+\epsilon  V: \ U\mapsto \Omega_\epsilon$ is biholomorphic for sufficiently small $\epsilon$ due to Theorem 2.4, \cite{Pascu}.

Then the following theorem holds true
\begin{theorem}
One has the following variational formula for the determinant of the Dirichlet boundary problem
\begin{equation}\label{WZ}
\frac{d \log {\rm det}\,\Delta_{\Omega_\epsilon}}{d \epsilon}\Big|_{\epsilon=0}=
\frac{1}{6\pi}\Re\bigg[\int_\Gamma V(R_\Omega^{-1}(z)) \overline{ \nu(z)}\left(\Re(\nu^2 \{R_\Omega^{-1}(z), z\})-k^2(z)\right)|dz|\bigg]\,
\end{equation}
where
$\nu(z)=|w'|w(w')^{-1}$  (with $w(z):=R^{-1}_\Omega(z)$) is the unit outward normal to $\Gamma$ at $z$, 
 $k(z)$ is the curvature of $\Gamma$ at $z$ and $\{\cdot, \cdot\}$ denotes the Schwarzian derivative.
\end{theorem} 

 \begin{rem} As a matter of fact, formula (\ref{WZ}) should be completely attributed to Wiegmann and Zabrodin. In (4.29) from \cite{WZ} they stated the variational formula for the determinant of the Dirichlet boundary problem under normal variations of the boundary.

Unfortunately, Wiegmann and Zabrodin do not give a proof of their wonderful formula referring to "a direct calculation". As the reader will see even in our (in a sort, simpler) case the direct calculation we use to prove (\ref{WZ}) is rather cumbersome.  

It should be, of course, acknowledged that without Wiegman-Zabrodin's insight it would be non-realistic to identify the $k^2$-term in the very long intermediate expression we get as a result of our calculation.  
\end{rem}   

{\bf Proof.}  First we collect some useful formulas that will be applied in the sequel. In what follows we surpress the inverse Riemann map in all the expressions. Say, $w$  always denotes either the standard coordinate in ${\mathbb C}\supset U$  or  a function $w=R^{-1}(z)$,  $z\in {\mathbb C}\supset \Omega$; in the latter case $w'$ denotes the complex derivative w. r. t. to $z$. Similarly $z$ is either a coordinate in ${\mathbb C}\supset \Omega$ or a function of $w\in {\mathbb C}\supset U$ defined by $z(w)=R(w)$; in the latter case $z'$ denotes the complex derivative w. r. t. $w$.  In particular, using these agreements,  we have  $z'=(w')^{-1}$. 
  
Let $w=e^{it}$ be the parametrization of $S^1=\partial U$. Then one has
 \begin{equation}\label{dt} dt=\frac{dw}{iw} \end{equation}
and
\[\overline{dw}=-\frac{dw}{w^2}\]
on $S^1$.
One has
\begin{equation}\label{dz} |dz|=\frac{dt}{|w'|}
\end{equation}
on $\Gamma$. 
The curvature $k$ of $\Gamma$ can be computed via the formula
\[k=|w'|\Re\bigg\{1+w\frac{z''}{z'}\bigg\}\]
and the unit (outward) normal to $\Gamma$ is given by
\[\nu=\frac{|w'|\,w}{w'}.\]

Relations $\nu^{-1}=\bar \nu$ and $w^{-1}=\bar w$ (on $\Gamma$) will be often used in the sequel. Let us also remind the reader the standard property
\begin{equation}\label{Sch}\{z, w\}(dw)^2=-\{w, z\}(dz)^2\end{equation} 
of the Schwarzian derivative. 

We start with the following version of the Alvarez formula (see, e. g., \cite{OPS}, eq-n (1.4)) for $\log {\rm det}\,\Delta_{\Omega_\epsilon}$:
\begin{equation}\label{Alvarez}
-12\pi\log {\rm det}\,\Delta_{\Omega_\epsilon}=C+\int_{S^1}\phi\partial_r\phi\,dt+2\int_{S^1}\phi\,dt,
\end{equation}
where $C$ is an absolute constant and
\begin{equation*}
\phi=\log|z'(w)+\epsilon V'(w)|=\Re \log z'(w)+\epsilon \Re \frac{V'}{z'}+O(\epsilon^2)\,.
\end{equation*}
One has
\begin{equation*}
\partial_r\phi=\partial_r\log|z'(re^{it})+\epsilon V'(re^{it}|=
\Re\bigg(\frac{wz''}{z'}\bigg)+\epsilon\Re\bigg\{\frac{wV''}{z'}-\frac{wz''V'}{z'^2}\bigg\}+O(\epsilon^2)
\end{equation*}
Thus, the coefficient near $\epsilon$ in the asymptotical expansion of the right hand side of (\ref{Alvarez}) reads as
\[\int_{S^1}\bigg[\Re\bigg(\frac{V'}{z'}\bigg)\Re\bigg(\frac{wz''}{z'}\bigg)+\Re\bigg\{\frac{wV''}{z'}-\frac{wz''V'}{(z')^2}\bigg\}\Re\log z'(w)]\bigg]\,dt+
2\Re\int_{S^1}\frac{V'}{z'}\,dt\]

Moving $\Re$ from factors containing $V, V'$ and $V''$ out of the integrals and then applying Cauchy formula to replace the real parts of holomorphic functions by the conjugate holomorphic, and  making use of $\bar w=w^{-1}$ and (\ref{dt}), one rewrites this as
\begin{align*}
-24\pi\frac{d \log {\rm det}\,\Delta_{\Omega_\epsilon}}{d \epsilon}\Big|_{\epsilon=0}=\\
=\Re\frac{1}{i}\Bigg[ \int_{S^1}V'\frac{1}{w^2z'}\overline{\bigg(\frac{z''}{z'}\bigg)}\,dw +\int_{S^1}V''\frac{1}{z'}\overline{\log z'}\,dw -
\int_{S^1}V' \frac{z''}{(z')^2} \overline{\log z'}\,dw+\\
+4\int_{S^1}V'\frac{1}{z'w}\,dw\Bigg]=\Re\frac{1}{i}[A+B-C+4D].
\end{align*}
Notice that $\overline{ dw}=-dw/w^2$ and, therefore, we have the following integration by parts formulas for three holomorphic functions $V$, $f$ and $g$:
\[\int_{S^1}V'f\bar g \,dw=-\int_{S^1}V f' \bar g\, dw+\int_{S^1}V\bigg(\frac{f}{w^2}\bigg)\overline{ g'}\,dw\]
and
\[\int_{S^1}V''f\bar g\,dw= \int_{S^1}V f'' \bar g\,dw - \int_{S^1}V\frac{f}{w^2}\overline{g'}\,dw -\int_{S^1}V\bigg(\frac{f}{w^2} \bigg)'\overline{g'}\,dw+ 
\int_{S^1}V\frac{f}{w^4}\overline{g''}\,dw\]
This gives
\[A=\int_{S^1}V\bigg(\frac{2}{w^3z'}+\frac{z''}{w^2(z')^2}\bigg)\overline{\bigg(\frac{z''}{z'}\bigg)}\,dw +\int_{S^1}V\frac{1}{w^4z'}\overline{(\bigg(\{z, w\}+\frac{1}{2}\bigg(
\frac{z''}{z'}\bigg)^2                           \bigg)    }\,dw,\]
\[B=\int_{S^1}V\bigg(\frac{1}{z'}\bigg)''\overline{\log z'}\,dw+\int_{S^1}V\frac{z''}{w^2(z')^2}\overline{\bigg(\frac{z''}{z'}\bigg) }dw+\]
\[+\int_{S^1}V\bigg(  \frac{z''}{(z')^2w^2} +\frac{2}{w^3z'}\bigg)\overline{\bigg( \frac{z''}{z'}          \bigg)     }dw+\int_{S^1}V\frac{1}{z'w^4}\overline{\bigg(  \{z, w\}+\frac{1}{2}\bigg(\frac{z''}{z'}    \bigg)^2                            \bigg)}dw,\] 
\[-C=\int_{S^1}V\bigg(\frac{z''}{(z')^2}\bigg)'\overline{\log z' }dw -\int_{S^1}V\frac{z''}{(z')^2w^2}\overline{\bigg(\frac{z''}{z'} \bigg)}dw
\]
and
\[4D=\int_{S^1}V\bigg(\frac{4z''}{w(z')^2}+\frac{4}{w^2z'}\bigg)dw\]
Summing all this up, with some effort, one arrives at
\[-24\pi\frac{d \log {\rm det}\,\Delta_{\Omega_\epsilon}}{d \epsilon}\Big|_{\epsilon=0}=\Re\frac{1}{i}\int_{S^1}V\bigg(\frac{2}{w^4z'}\overline{\{z, w\}}+\frac{1}{w^4z'}\overline{(\frac{z''}{z'}\bigg)^2}+$$ $$+\bigg(\frac{4}{w^3z'}+ \frac{2z''}{w^2(z')^2}\bigg)
\overline{\bigg(\frac{z''}{z'}\bigg)}+\frac{4z''}{w(z')^2}+\frac{4}{w^2z'}
\bigg)dw=\]
\[=\Re\frac{1}{i}\int_{S^1}\frac{V}{w^2z'}\bigg\{2\bar w^2\overline{\{z, w\}}+\Bigg[w^2\bigg(\frac{z''}{z'}\bigg)^2\Bigg]+\bar w^2\overline{\bigg(\frac{z''}{z'}\bigg)^2}+4\bar w
\overline{\bigg(\frac{z''}{z'}\bigg)}+$$$$+4w\frac{z''}{z'}+4+2\frac{z''}{z'}\overline{\bigg(\frac{z''}{z'}\bigg) }          
\bigg\}dw,\]
where the term in the big square brackets is inserted by force (it integrates to zero by the Cauchy theorem).
It remains to notice that
the term $2\bar w^2\overline{\{z, w\}}$ in the last expression can be replaced by $4\Re (w^2{\{z, w\}})$ due to the Cauchy theorem,  find out that, amazingly,  
\[\frac{4k^2}{|w'|^2}=4+w^2\bigg(\frac{z''}{z'}\bigg)^2+(\bar w)^2\overline{\bigg(\frac{z''}{z'}\bigg)^2}+4w\frac{z''}{z'}+4\bar w\overline{\bigg(\frac{z''}{z'}\bigg)}+
2\frac{z''}{z'}\overline{\bigg(\frac{z''}{z'}\bigg)},\]
pass from the integration over $S^1$ to the integration over $\Gamma$, make use of (\ref{Sch}), (\ref{dt}), (\ref{dz}) and 
the relation 
\[|dz|=\frac{dw}{iw|w'|},\]
and, finally, get (\ref{WZ}).

\subsection*{Acknowledgements}
The first author thanks Max Planck Institute for Mathematics in Bonn for warm hospitality and excellent working conditions.

\subsection*{Funding}
The research of the second author was supported by  Fonds de recherche du Qu\'ebec.

\end{document}